\documentclass[10pt]{amsart}
\usepackage{latexsym,amssymb,amsmath}
\usepackage[margin=1in]{geometry}
\usepackage{amsmath,amsfonts,xcolor}
\usepackage{epsfig}
\usepackage{graphicx}
\usepackage{verbatim}
\usepackage{hyperref}
\hypersetup{
  colorlinks   = true,    
  urlcolor     = blue,    
  linkcolor    = blue,    
  citecolor    = purple      
}
\newtheorem{thm}{Theorem}[section]
\newtheorem{prop}[thm]{Proposition}
\newtheorem{lem}[thm]{Lemma}

\newtheorem{definition}[thm]{Definition}
\newtheorem{rmk}[thm]{Remark}

\newcommand{\R}{\mathbb{R}}

\newcommand{\C}{\mathbb{C}}

\newcommand{\lb}{\langle}
\newcommand{\rb}{\rangle}

\renewcommand{\t}{\widetilde}
\renewcommand{\d}{\,\operatorname{d}\!}

\renewcommand{\Re}{\operatorname{Re}}
\renewcommand{\Im}{\operatorname{Im}}
\renewcommand{\hat}{\widehat}

\allowdisplaybreaks

\title{Well-posedness for the Schr\"odinger-KdV system on the half-line}

\author{E. ~Compaan, W. ~Shin, N.~Tzirakis}

\address{Department of Mathematics, University of Illinois, Urbana, IL 61801, U.S.A.}
\email{e.compaan@gmail.com}
\email{wshin14@illinois.edu}
\email{tzirakis@illinois.edu}


\date{}

\begin{document}

\begin{abstract} 
In this paper we obtain improved local well--posedness results for the  Schr\"odinger-KdV system on the half-line. We employ the Laplace--Fourier method in conjunction with the restricted norm method of Bourgain appropriately modified in order to accommodate the bounded operators of the half--line problem. Our result extends the previous local results in \cite{CC}, \cite{CC1} and \cite{HY} matching the results that Wu, \cite{Wu}, obtained for the real line system. We also demonstrate the uniqueness for the full range of locally well--posed solutions. In addition we obtain global well--posedness on the half--line for the energy solutions with zero boundary data, along with polynomial--in--time bounds for higher order Sobolev norms for the Schr\"odinger part. 
\end{abstract}

\maketitle

\section{Introduction}
In this paper we study the initial-boundary value problem for the Schr\"odinger--KdV system posed on the positive half line
\begin{align}\label{eq: SKdV}
\begin{cases}
i\partial_{t}u + \partial_{x}^{2}u = \alpha uv + \beta |u|^2u,  \\
\partial_{t}v +  \partial_{x}^{3}v + v \partial_{x}v = \gamma \partial_{x}(|u|^2), \\
u(x,0) = u_0(x), \quad v(x,0) = v_0(x),\\
u(0,t) = f(t), \quad v(0,t) = g(t),
\end{cases} & (x,t) \in \R^+ \times \R^+,
\end{align}
where $u=u(x,t) \in \C$, $v=v(x,t) \in \R$, and $\alpha, \beta, \gamma$ are real constants. 

The system on the real line models the nonlinear interaction of a short wave modeled by the Schr\"odinger part and a long wave modeled by the KdV part. The system arises in several fields of physics and fluid dynamics. The case $\beta \neq 0$ discribes non-resonant interactions while the case $\beta=0$ describes resonant interactions. For a detailed discussion of the applications we cite the papers \cite{Wu}, \cite{CC}, \cite{HY} and the references therein. We should note that although the decoupled equations, namely the cubic nonlinear Schr\"odinger equation and the KdV equations are completely integrable the coupled system is not. In our work we study the system in the half--line interval imposing boundary conditions to the solution operators. From our point of view the model we study corresponds to the case where a wave generator acting as a fixed wave source, produces a disturbance in such a way that a wave motions are generated and then evolve into a medium in only one direction where the motion is unrestricted. 

Our goal is to improve the previous local well--posedness results in \cite{CC}, \cite{CC1} and \cite{HY}, and match the best possible results that are available on the real line, \cite{Wu}, by utilizing recent techniques that have been developed to study nonlinear PDE on semi--infinite domains. Our method is based on the restricted norm method ($X^{s,b}$--method) of Bourgain. Since the dispersive effects of the system on the half--line are weaker than the full real line system, one has to set up a modification of the $X^{s,b}$ method to accomodate the presence of the bounded operators. Our method is also inspired by previous developments in IBVP by Bona and his collaborators \cite{bsz},\cite{bsz2}, Kenig and Colliander in \cite{CK}, and Holmer \cite{Holm}, \cite{Holm1}. Our approach is most closely related to those in \cite{bsz} and \cite{CK}. Regarding the Schr\"odinger--KdV system, the first attempt to obtain a well--posedness theory for the half--line in \cite{CC} used the Colliander--Kenig method. The method is based on writting an implicit solution--formula with certain terms enforcing the boundary conditions. The drawback of the method is that uniqueness cannot be proved with the tools in use. The method was developed further in the Ph.D. thesis of Holmer in \cite{Holm}, \cite{Holm1}. For more applications of this method on certain initial--boundary value problems the reader can consult \cite{cm}. 

Another method that has been employed for the resolution of IBVP is the unified transform method of Fokas, introduced in \cite{fokas}. See also \cite{fhm} and the references therein for more recent examples and especially for applications of the method to nonlinear problems. The method is quite general and works for a variety of linear and nonlinear problems that are posed on complicated domains. This is the method the authors employ in \cite{HY}.

In the present manuscript we improve the well--posedness theory for the semi--infinite interval for $L^2$ based solutions by employing multilinear harmonic analysis techniques. We use the methods developed in \cite{bsz} and \cite{ET} where the authors extended the initial data of certain dispersive PDE posed on $\R^+$ to the whole real line and then employed the Laplace and Fourier transform method to solve the linear problem in the usual Sobolev spaces. After this an equivalent integral representation (Duhamel's formula) of the solution is written using the trace of the linear boundary operator and the nonlinearity. Trace theorems dictate the regularity of the boundary functions, while the nonlinearity of each particular PDE dictates the Banach space in which a fixed point of the integral equation is sought. Our method has the advantage of establishing that the nonlinear evolution is smoother than the initial data. These smoothing estimates enable us to conclude that our solutions are unique in the sense that they are independent of a particular extension of the initial data. 
 We should note that for IBVP, uniqueness of solutions is not straightforward at all, especially in a low regularity regime where there are no conservation laws or \emph{a priori} bounds that are derived from the differential form of the conserved quantities. In this paper by proving a priori bounds for smooth solutions we obtain uniqueness for higher regularities. Persistence of regularity estimates then allows us to obtain uniqueness for rougher solutions. Persistence of regularity for the system (\ref{eq: SKdV}) can be obtained after one has established certain smoothing estimates. We discuss this issue in Section \ref{section: persistence}.

In the context of our work, well--posedness means existence and uniqueness of strong solutions \cite{tc} along with continuity of the data to solution map, at least locally in time (for a more precise definition see Section \ref{section: notion of a sol}). With the notations defined in Section \ref{section: notation}, we now state our main results..

\begin{definition}
We say that $(s,k)$ is admissible if 
\begin{align*}
\begin{array}{lr}
        0\leq s<\frac{5}{2}, \; s \neq \frac12 \; & \text{if } \beta \neq 0,\\
        -\frac{3}{16}<s<\frac{5}{2}, \; s \neq \frac12 \; & \text{if } \beta = 0,
\end{array}
\end{align*}
$-\frac{3}{4}<k<\frac{7}{2}$, $k \neq \frac{1}{2}$ and $s-2<k<\min(4s,s+1)$. 
\end{definition}

\begin{thm}\label{LWP}
For an admissible pair $(s,k)$, the equation (\ref{eq: SKdV}) is locally well-posed in $H^{s}(\R^{+}) \times H^{k}(\R^{+})$ in the sense of Definition \ref{def: LWP}.
Moreover, for any
\begin{align*}
\begin{array}{lr}
        0\leq a <\frac12\min(k-s+2, k+1, 4s, 5-2s, 1), & \text{if } \beta \neq 0,\\
        0\leq a <\frac12\min(k-s+2, k+1, 5-2s, 1), & \text{if } \beta = 0 \text{ and } s \geq 0, \\
        0\leq a <\frac12\min(s+k+\frac32, k+1, 5-2s, 1), & \text{if } \beta = 0 \text{ and } s < 0,
\end{array}
\end{align*}
we have the smoothing for the Schr\"odinger part
\begin{align} \label{smoothing}
u-S_{0}^{t}(u_{0},f) \in C_{t}H^{s+a}_{x}([0,T_{max}) \times \R^{+}),
\end{align}
where $T_{max}$ is the lifetime of $(u,v)$.
\end{thm}
For $s<\frac52$, $k <\frac72$ and $s, k \neq \frac{1}{2}$, our well-posedness region matches that of \cite{Wu} up to the boundary, which is so far the best known result. We do not expect nonlinear smoothing for the KdV part, as it is known, \cite{IMT}, that the nonlinear evolution of KdV equation on $\R$ is not smoother than the initial data. 

\begin{thm}\label{Global}
    Let $(s,k)$ be an admissible pair satisfying $s \geq 1$ and $k \geq 1$. Given $(u_0,f,v_0,g) \in \mathcal{H}^{s,k}$, the solution $(u,v)$ to the system (\ref{eq: SKdV}) has a global-in-time extension provided $\alpha \gamma >0$, $f=0$ and $g=0$. Moreover, $\|u\|_{H^s(\R^+)}$ can grow at most polynomially-in-time.
\end{thm}
In \cite{CC1}, the authors proved global well-posedness for (\ref{eq: SKdV}) with zero boundary condition in the energy space $(s,k)=(1,1)$. We extend this result for all $s \geq 1$ and $k \geq 1$ using persistence of regularity. We note this and the polynomial growth result are all based on smoothing estimates.

Now we outline the organization of the paper. In Section \ref{section: notation}, we define several function spaces to be used in this paper and briefly discuss their basic properties. In Section \ref{section: notion of a sol}, we define the notion of a solution by reformulating the equation (\ref{eq: SKdV}) in terms of Duhamel's formula involving boundary terms. In Section \ref{section: A priori estimates}, we state several linear and nonlinear estimates needed to close the contraction argument. In Section \ref{section: Main results} we present the proofs of our main results. First, we show the existence of local solutions via Banach's fixed point theorem. Then we utilize the smoothing estimate demonstrated in Section \ref{section: A priori estimates} to prove the persistence of regularity property of solutions. Using this, we show the uniqueness of low-regularity solutions. Finally, we discuss nonlinear smoothing and the growth of the Sobolev norm of solutions. In Section \ref{section: high regularity uniqueness}, we prove the uniqueness of high-regularity solutions by Gronwall's inequality. Finally, we present the proofs of key multilinear estimates in Section \ref{section: proofs of estimates}. A standard Lemma that is used throughout the paper is stated separately in the Appendix.

\section{Notation \& Function spaces} \label{section: notation}

\subsection{Notation}

Define the one-dimensionl Fourier transform as
\begin{align*}
    \hat{f}(\xi)=\mathcal{F}f(\xi)=\int_{\R}e^{-ix\xi}f(x)dx.
\end{align*} 
If $u(x,t)$ is a function on $\R \times \R$ we define its space-time Fourier transform $\hat{u}(\xi,\tau)$ by
\begin{align*}
    \hat{u}(\xi,\tau)=\mathcal{F}u(\xi,\tau)=\int_{\R \times \R}e^{-i(x\xi+t\tau)}u(x,t)dx dt.
\end{align*} 
Let $S_{\R}^{t}u_0$ denote the solution to the linear Schr\"odinger equation on $\R$, i.e. $S_{\R}^{t}u_0$ solves
\[ \begin{cases} i\partial_{t}u + \partial_{x}^{2}u = 0, \quad (x,t) \in \R\times \R^+, \\
                 u(x,0) = u_0(x). \end{cases} \]
For $\R^{+}=(0,\infty)$, let $S^t_0(u_0,f)$ denote the solution to the corresponding linear initial-boundary value problem (IBVP) on $\R^+$:
\[
\begin{cases}
i\partial_{t}u + \partial_{x}^{2}u= 0, \quad (x,t) \in \R^+ \times \R^+, \\
u(x,0) = u_0(x), \quad u(0,t) = f(t). 
\end{cases}
\]
Similarly, let $W_{\R}^{t} v_0$ and $W^t_0(v_0,g)$ denote the solutions to the linear Airy equation on $\R$ and the linear Airy IBVP on $\R^+$,
\[
\begin{cases} 
\partial_{t}v +  \partial_{x}^{3}v = 0, \quad (x,t) \in \R \times \R^+, \\
v(x,0) = v_0(x), \end{cases}
 \text{and} \quad
\begin{cases}
\partial_{t}v +  \partial_{x}^{3}v = 0, \quad (x,t) \in \R^+ \times \R^+, \\
v(x,0) = v_0(x), \quad v(0,t) = g(t), \end{cases}\]
respectively. 

For a space-time function $f$, we denote 
\begin{align*}
    D_{0}f(t)=f(0,t).
\end{align*}
We denote by $\chi := \chi_{(0,\infty)}$ the indicator function of $(0,\infty)$. Let $\lb \xi \rb :=(1+|\xi|^2)^{1/2}$.

\subsection{Function spaces}
Define the Sobolev space $H^{s}(\R)$ by the norm 
\begin{align*}
    \Vert f \Vert_{H^{s}(\R)}=\left (\int_{\R}\lb \xi \rb^{2s}|\hat{f}(\xi)|^{2} d\xi \right )^{\frac{1}{2}}
\end{align*}
and its restriction on $\R^{+}=(0,\infty)$ by the norm
\begin{align*}
    \Vert f \Vert_{H^{s}(\R^{+})}=\inf\{\Vert g \Vert_{H^{s}(\R)}: g = f \text{ on } \R^{+}\}.
\end{align*}
Define the Schr\"odinger Fourier restriction space $X^{s,b}$ and the KdV Fourier restriction space $\t X^{k,b}$ by the following norms:
\begin{align*}
\| u \|_{X^{s,b}} = \left\|\lb \xi \rb^{s} \lb \tau + \xi^2 \rb^b \hat{u}(\xi, \tau) \right\|_{L^2_{\xi,\tau}},\\
\| v \|_{\t X^{k,b}} = \left\|\lb \xi \rb^{k} \lb \tau - \xi^3 \rb^b \hat{v}(\xi, \tau) \right\|_{L^2_{\xi,\tau}}.
\end{align*}
 To handle the boundary terms, we are forced to choose $b<\frac12$. On the other hand, it is known that the key bilinear estimate $\|\partial_{x}(v_1v_2)\|_{\t X^{k,b}} \lesssim \| v_1\|_{\t X^{k,b}}\| v_2\|_{\t X^{k,b}}$ fails for all $b<\frac12$. To rectify this, we need the low-frequency modification space $D_\alpha$ introduced in \cite{CK} defined by the norm
\[ \| v \|_{D_\alpha} = \| \chi_{[-1,1]}(\xi) \lb \tau \rb^\alpha \hat{v}(\xi,\tau)\|_{L^2_{\xi,\tau}}.\]
We shall write $\t X^{k,b}_\alpha$ for $\t X^{k,b}\cap D_\alpha$. Note that
\begin{align*} \| v \|_{\t X^{k,b}_\alpha} &\approx \left\| \left(\lb \xi \rb^{k} \lb \tau - \xi^3 \rb^b + \chi_{|\xi| < 1} \lb \tau \rb^\alpha \right) \hat{v}(\xi,\tau) \right\|_{L^2_{\xi,\tau}}\\
&\approx \left\| \lb \xi \rb^{k} \left(\lb \tau - \xi^3 \rb^b + \chi_{|\xi|<1} \lb \tau \rb^\alpha \right) \hat{v}(\xi,\tau)\right\|_{L^2_{\xi,\tau}}.
\end{align*}
 For our persistence of regularity argument, we need our resolution space to be embedded in $C_tH^s_x \times C_tH^k_x$. However, since the space $X^{s,b} \times \t X^{k,b}_\alpha$ does not embed in $C_tH^s_x \times C_tH^k_x$ for any $b<\frac12$, we need the modified spaces $Y^{s,b}$ and $\t Y^{k,b}_{\alpha}$ as introduced in \cite{Bourgain}, defined by the following norms:
\begin{align*}
\| u \|_{Y^{s,b}} =\| u \|_{X^{s,b}}+\|\lb \xi \rb^{s}\hat{u}(\xi, \tau) \|_{L^2_{\xi}L^{1}_{\tau}}, \\
\| u \|_{\t Y^{k,b}_{\alpha}} =\| u \|_{\t X^{k,b}_{\alpha}}+\|\lb \xi \rb^{k}\hat{u}(\xi, \tau) \|_{L^2_{\xi}L^{1}_{\tau}}.
\end{align*} 
For any time interval $I$, define the time-localized space $Y^{s,b}(I)$ by the norm
\begin{align*}
    \| u \|_{Y^{s,b}(I)}=\inf\{\| v \|_{Y^{s,b}}:v=u \text{ on } \R \times I \}.
\end{align*}
We analogously define the space $\t Y^{k,b}_{\alpha}(I)$. Finally, $\mathcal{H}^{s,k}$ will denote the set of all elements $(u_0,f,v_0,g) \in H^{s}(\R^+) \times H^{\frac{2s+1}{4}}(\R^+) \times H^{k}(\R^+) \times H^{\frac{k+1}{3}}(\R^+)$ satisfying the additional compatibility conditions $u_{0}(0) = f(0)$ for $s > \frac{1}{2}$ and $v_{0}(0) = g(0)$ for $k > \frac{1}{2}$.

We collect some basic properties of the above spaces that will be used later. 
\begin{lem} [{\cite{Bourgain}}] \label{embedding}
    For any $s, k \in \R$, we have inclusions $Y^{s,b}(I)\hookrightarrow C_{t}H^{s}_{x}(I \times \R)$ and $\t Y^{k,b}_{\alpha}(I)\hookrightarrow C_{t}H^{k}_{x}(I \times \R)$ with inequalities
    \begin{align*}
        \| u \|_{C_{t}H^{s}_{x}(I \times \R)} \leq \| u \|_{Y^{s,b}(I)} \quad \text{and} \quad \| u \|_{C_{t}H^{k}_{x}(I \times \R)} \leq \| u \|_{\t Y^{k,b}_{\alpha}(I)}.
    \end{align*}
\end{lem}

\begin{lem}[{\cite[Lemma 2.11]{Tao}}]\label{lemma: X^s,b scaling}
For any $T>0$ and $-\frac{1}{2}<b_1<b_2<\frac{1}{2}$, we have
    \begin{align*}
    \|\eta(t/T)u\|_{X^{s,b_1}} \lesssim T^{b_{2}-b_{1}}\|\eta(t/T)u\|_{X^{s,b_2}}
    \end{align*}
    and
    \begin{align*}
        \|\eta(t/T)u\|_{\t X^{k,b_1}} \lesssim T^{b_{2}-b_{1}}\|\eta(t/T)u\|_{\t X^{k,b_2}}.
    \end{align*}
\end{lem}

\begin{lem} \label{lemma: concatenation}
Let $I_1$ and $I_2$ be nonempty open intervals such that $I_1 \cap I_2 \neq \varnothing$. Then for any $s,k,b,\alpha \in \R$, we have
\begin{align*}
   \| u \|_{Y^{s,b}(I_1 \cup I_2)} \lesssim \| u \|_{Y^{s,b}(I_1)}+\| u \|_{Y^{s,b}(I_2)}
\end{align*}
and
\begin{align*}
   \| u \|_{\t Y^{k,b}_{\alpha}(I_1 \cup I_2)} \lesssim \| u \|_{\t Y^{k,b}_{\alpha}(I_1)}+\| u \|_{\t Y^{k,b}_{\alpha}(I_2)}.
\end{align*}
\end{lem}
\begin{proof}
Let $I$ be an open time interval, and $\eta(t)$ be a Schwartz function supported on another open interval $J$. We claim that
 \begin{align} \label{eq: time localization}
        \| \eta(t) u \|_{Y^{s,b}(I)}\lesssim \| u \|_{Y^{s,b}(I \cap J)}. 
\end{align}
By \cite[Theorem 2.11]{Tao} and Young's convolution inequality, we have $\| \eta(t) u \|_{Y^{s,b}(I)} \lesssim \| u \|_{Y^{s,b}}$. Thus 
    \begin{align*}
        \| \eta(t) u \|_{Y^{s,b}(I)}=\inf_{u=\t u \text{ on } I \cap J}\|\eta(t) \t u\|_{Y^{s,b}(I)}\lesssim \inf_{u=\t u \text{ on } I \cap J}\| \t u \|_{Y^{s,b}}=\| u \|_{Y^{s,b}(I \cap J)}. 
    \end{align*}
Finally, by considering smooth cutoff functions $\eta_1, \, \eta_2$ such that $\eta_1 + \eta_2=1$ on $I_1 \cup I_2$ and $\operatorname{supp}(\eta_1) \cap (I_2 \setminus I_1)=\operatorname{supp}(\eta_2) \cap (I_1 \setminus I_2) = \varnothing$, the inequality
\begin{align*}
   \| u \|_{Y^{s,b}(I_1 \cup I_2)} \lesssim \| u \|_{Y^{s,b}(I_1)}+\| u \|_{Y^{s,b}(I_2)}
\end{align*}
readily follows from the triangle inequality and (\ref{eq: time localization}). We can similarly prove the second inequality.
\end{proof}

\section{Notion of a solution} \label{section: notion of a sol}

We will construct solutions to (\ref{eq: SKdV}) first by solving some integral equation on the entire real line, and then restricting the corresponding solutions to the half line. For this, we will need the following lemma regarding the extension of Sobolev functions defined on $\R^+$:
\begin{lem} \cite{ET} \label{lemma: extension}
    Let $h\in H^{s}(\R^{+})$ for some $-\frac{1}{2} < s <\frac{3}{2}$. \\
    (i) If $-\frac{1}{2}<s<\frac{1}{2}$, then $\Vert \chi_{(0,\infty)}h\Vert_{H^{s}(\R)} \lesssim \Vert h \Vert_{H^{s}(\R^{+})}$. \\
    (ii) If $\frac{1}{2}<s<\frac{3}{2}$ and $h(0)=0$, then $\Vert \chi_{(0,\infty)}h\Vert_{H^{s}(\R)} \lesssim \Vert h \Vert_{H^{s}(\R^{+})}$.
\end{lem}
To construct the solutions to (\ref{eq: SKdV}) we begin by extending $S^t_0(0,f)$ and $W^t_0(0,g)$ to all $x \in \R$. We use the following explicit formulas:
\begin{align*}
    S^t_0(0,f)&=\frac{1}{\pi}\int_{0}^{\infty}e^{-i\beta^{2}t+i\beta x}\beta \hat{\chi f}(-\beta^{2})d\beta+\frac{1}{\pi}\int_{0}^{\infty}e^{i\beta^{2}t-i\beta x}\beta \hat{\chi f}(-\beta^{2})d\beta \\
    &:=S_{1}f+S_{2}f,
\end{align*}
\begin{align*}
    W_{0}^{t}(0,g)=\frac{3}{\pi}Re\int_{0}^{\infty} e^{i\beta^{3}t}e^{-i \beta x\cos(\frac{\pi}{3})}e^{-\beta x\sin(\frac{\pi}{3})}\widehat{\chi g}(\beta^{3})\beta^{2}d\beta,
\end{align*}
see \cite{CT, ET}. Let $\rho \in C^{\infty}(\mathbb{R})$ be a function supported on $(0,\infty)$ such that $\rho$=1 on $[0,\infty)$. We further assume that
\begin{align} \label{eq: mean zero extension}
    \int_{-1}^{\infty}e^{-i y\cos(\frac{\pi}{3})}e^{-y\sin(\frac{\pi}{3})}\rho(y)dy=0.
\end{align} 
We extend $S^t_0(0,f)$ and $W_{0}^{t}(0,g)$ to all $x \in \R$ by letting 
\begin{align*}
    S_{2}f(x,t)=\frac{1}{\pi}\int_{0}^{\infty}e^{i\beta^{2}t-ix}\rho(\beta x)\beta \hat{\chi f}(-\beta^{2})d\beta
\end{align*}
and
\begin{align*}
   W_{0}^{t}(0,g)=\frac{3}{\pi}Re\int_{0}^{\infty} e^{i\beta^{3}t}e^{-i \beta x\cos(\frac{\pi}{3})}e^{-\beta x\sin(\frac{\pi}{3})}\rho(\beta x)\widehat{\chi g}(\beta^{3})\beta^{2}d\beta.
\end{align*}

For $(u_{0},v_{0})\in H^s(\R^+) \times H^k(\R^+)$, let $(u_{0e},v_{0e})$ be an extension of $(u_{0},v_{0})$ to the full line satisfying $\| u_{0e}\|_{H^s(\R)} \lesssim \| u_{0}\|_{H^s(\R^+)}$ and $\| v_{0e}\|_{H^k(\R)} \lesssim \| v_{0}\|_{H^k(\R^+)}$. Note that the existence of this extension is assured by Lemma \ref{lemma: extension}. Consider the system of integral equations
\begin{align}\label{eq: SKdVintegral}
\begin{cases}
u=S_{\R}^{t}u_{0e}-i\int_{0}^{t}S_{\R}^{t-t'}F(u,v)dt'+S_{0}^{t}(0,f-p-q), \\
v=W_{\R}^{t}v_{0e}+\int_{0}^{t}W_{\R}^{t-t'}G(u,v)dt'+W_{0}^{t}(0,g-\t p-\t q),
\end{cases}
\end{align}
where
\begin{multline*}
    F(u,v)=\alpha uv+\beta|u|^{2}u, \; p(t)=D_{0}(S_{\R}^{t}u_{0e}), \;
    q(t)=D_{0}\left ( \int_{0}^{t}S_{\R}^{t-t'}F(u,v)dt'\right ), \\
    G(u,v)=\gamma(|u|^{2})_{x}-vv_{x}, \;
    \t p(t)=D_{0}(W_{\R}^{t}v_{0e}), \; \t q(t)=D_{0}\left ( \int_{0}^{t}W_{\R}^{t-t'}G(u,v)dt'\right ).
\end{multline*}
\begin{definition} \label{def: LWP}
    We say the equation (\ref{eq: SKdV}) is locally well--posed in $H^{s}(\R^{+}) \times H^{k}(\R^{+})$ if for any $(u_0,f,v_0,g) \in \mathcal{H}^{s,k}$, the system (\ref{eq: SKdVintegral}) has a unique solution $(u,v)$ in 
    \begin{multline*}
        \left [Y^{s,b}([0,T]) \cap C_t H^s_x([0,T] \times \R)\cap C_{x}H^{\frac{2s+1}{4}}_{t}(\R \times [0,T]) \right ] \\ \times \left [\t Y^{k,b}_{\alpha}([0,T]) \cap C_t H^k_x([0,T] \times \R)\cap C_{x}H^{\frac{k+1}{3}}_{t}(\R \times [0,T]) \right]
    \end{multline*}
for some $T>0$, $\alpha>\frac{1}{2}$, and any $b < \frac{1}{2}$ sufficiently close to $\frac12$. Furthermore, the solution map $(u_{0},f,v_{0},g) \mapsto (u,v)$ is continuous from $\mathcal{H}^{s,k}$ to the above space.
\end{definition}

\section{A priori estimates} \label{section: A priori estimates}
In this section, we list some estimates which will be useful to bound the terms in (\ref{eq: SKdVintegral}).

\subsection{Linear estimates}
We collect some fundamental estimates on the linear evolutions $S_{\R}^t$, $W_{\R}^t$, $S_{0}^t$, and $W_{0}^t$.  
\begin{lem}[{\cite[Lemma 2.8]{Tao}, \cite[Lemma 5.2]{CK}}]\label{estimate: Free Solutions}
For any $s, k\in\R$ and any $b \in \R$, we have 
\[ \| \eta(t)S_{\R}^{t}u_0\|_{Y^{s,b}} \lesssim \| u_0\|_{H^{s}_x(\R)} \quad\text{and}\quad \| \eta(t)W_{\R}^{t} v_0 \|_{\t Y^{k,b}_{\alpha}} \lesssim \| v_0 \|_{H^{k}_x(\R \times [0,T])}.\]  
\end{lem}

\begin{lem}[{\cite[Lemma 3.1]{ET}, \cite[Lemma 4.1]{CK}}]\label{lemma: Kato Smoothing}
For any $s \geq -\frac12$ and any $k\geq -1$, we have 
\[\eta(t)S_{\R}^{t} u_0 \in C_{x}H^{\frac{2s+1}{4}}_t(\R \times \R) \quad \text{and} \quad \eta(t)W_{\R}^{t}v_0 \in C_{x}H^{\frac{k+1}3}_t(\R \times \R)\]
and we have
\[\| \eta(t)S_{\R}^{t} u_0 \|_{C_x H^{\frac{2s+1}{4}}_t(\R \times \R)} \lesssim \|u_0\|_{H^s_x(\R)} \quad \text{and} \quad \|  \eta(t)W_{\R}^{t}v_0 \|_{C_x H^{\frac{k+1}3}_t(\R \times \R))} \lesssim \| v_0\|_{H^k_x(\R)}.\]
\end{lem}
We remark that the above Kato smoothing result is stated in \cite{ET} for $s \geq 0$, but it holds for $s \geq -\frac12$ by the same proof. 

\begin{lem}\label{linHs} 
For any $s > -\frac{1}{2}$ and $f$ such that $\chi f \in H^\frac{2s+1}{4}(\R)$, we have 
 \begin{align*}
  S^t_0(0,f) \in C_tH^{s}_x( \R \times \R) \quad \text{and} \quad
  \eta(t) S^t_0(0,f) \in C_x H^{\frac{2s+1}{4}}_t( \R \times \R).
 \end{align*}
For any $k \geq -1$ and $g$ such that $\chi g \in H^{\frac{k+1}{3}}(\R)$, we have 
 \begin{align*}
  W^t_0(0,g) \in C_tH^{k}_x(\R \times \R) \quad \text{and} \quad
  \eta(t) W^t_0(0,g) \in C_x H^{\frac{k+1}{3}}_t( \R \times \R).
 \end{align*}
Furthermore, under the same assumptions, we have the estimates
 \begin{align*}
 \| S^t_0(0,f)\|_{C_tH^{s}_x( \R \times \R)} + \| \eta(t) S^t_0(0,f) \|_{C_x H^{\frac{2s+1}{4}}_t( \R \times \R)} \lesssim \| \chi f \|_{H^{\frac{2s+1}{4}}_t(\R)}
 \end{align*}
 and
  \begin{align*}
 \| W^t_0(0,g)\|_{C_tH^{k}_x( \R \times \R)} + \| \eta(t) W^t_0(0,g) \|_{C_x H^{\frac{k+1}{3}}_t( \R \times \R)} \lesssim \| \chi g \|_{H_t^{\frac{k+1}{3}}(\R)}.
 \end{align*}
\end{lem}
\begin{proof}
All above assertions are already proved in \cite[Lemma 3.2]{ET} and \cite[Lemma 4.6]{CT}, except for the statement about $S_{0}^{t}(0,f)$ when $s \in (-\frac{1}{2},0)$. For this case, we follow the argument given in the proof of \cite[Lemma 3.2]{ET}. Let $h(x)=e^{-x}\rho(x)$, where $\rho(x)$ is a smooth function supported on $(-1,\infty)$ and $\rho(x)=1$ for $x>0$. We only show that 
\begin{align*}
    Tg(x):=\int h(\beta x)\hat{g}(\beta)d\beta
\end{align*}
is bounded in $H^s$ for $s\in(-\frac{1}{2},0)$.
\begin{multline*}
    \Vert Tg \Vert_{H^s} =\Vert \lb \xi \rb^{s} \widehat{Tg} \Vert_{L^{2}_{\xi}} 
    =\left \Vert \lb \xi \rb^{s} \int \frac{1}{\beta} \hat{h}\left(\frac{\xi}{\beta} \right) \hat{g}(\beta)d\beta \right \Vert_{L^{2}_{\xi}} \\
    =\left \Vert \lb \xi \rb^{s} \int \frac{1}{\mu} \hat{h}(\mu) \hat{g}\left(\frac{\xi}{\mu} \right)d\mu \right \Vert_{L^{2}_{\xi}} 
    \leq \int \frac{1}{|\mu|} |\hat{h}(\mu)| \left \Vert \lb \xi \rb^{s}\hat{g}\left(\frac{\xi}{\mu} \right)\right \Vert_{L^{2}_{\xi}}d\mu.
\end{multline*}
Since $s < 0$, we have
\begin{align*}
    \left \Vert \lb \xi \rb^{s}\hat{g}\left(\frac{\xi}{\mu} \right)\right \Vert_{L^{2}_{\xi}} &= |\mu|^{\frac{1}{2}}\left\{ \int \lb \mu \xi \rb^{2s}|\hat{g}(\xi)|^{2}d\xi \right\}^{\frac{1}{2}}
    \leq 
    \begin{cases}
    |\mu|^{\frac{1}{2}} \Vert g \Vert_{H^{s}} & |\mu|>1 \\
    |\mu|^{\frac{1}{2}+s} \Vert g \Vert_{H^{s}} & |\mu| \leq 1.
    \end{cases}
\end{align*}
and the desired result follows.
\end{proof}

\begin{lem} \label{lemma: S_0^t(0,f) estimate}
Let $s >-\frac12$. Then for any $f$ satisfying $\chi_{(0,\infty)}f \in H^{\frac{2s+1}{4}}(\mathbb{R})$ and $b \in (\frac{1}{6},\frac{1}{2}]$, we have 
\[
\| \eta(t) S^t_0(0,f) \|_{Y^{s,b}} \lesssim \| \chi_{(0,\infty)} f \|_{H^{\frac{2s+1}{4}}_t(\R)}.\]
\end{lem}

\begin{proof} The estimate for the $X^{s,b}$-norm is proved in {\cite[Proposition 3.3]{ET}}, with the exception of $s \in (-\frac12, 0)$ case. For this case, we follow the argument given in the proof of \cite[Proposition 3.3]{ET}. The argument from that proof for the $S_1$ term holds for negative $s$ also. For the $S_2$ term, we study $\lb \partial \rb^{s} \left( \eta S_2 f \right)$, where $s \in (-\frac12, 0)$. By interpolation, it suffices to consider this case.

Let $h(x)=e^{-x}\rho(x)$, where $\rho(x)$ is a smooth function supported on $(-1,\infty)$ and $\rho(x)=1$ for $x>0$. Notice that 
\[ \mathcal{F}\left[\lb \partial \rb^{s} \left( \eta S_2 f \right)\right](\tau,\xi)  =  \int_0^\infty \hat{\eta}(\tau - \beta^2) \left[ \frac{\lb \xi \beta^{-1}\rb}{\lb \xi \rb} \right]^{-s}  \hat{ \lb \partial \rb^{s} h }(\xi \beta^{-1}) \hat{f}(\beta^2) \d \beta . \] 
In the case that $|\beta| \lesssim |\xi|$ or $|\beta| \lesssim 1$, the multiplier $\left [ \frac{\lb \xi \beta^{-1}\rb}{\lb \xi \rb} \right ]^{-s} $ of order at most $\beta^{s}$. In this case, the argument can be completed exactly as in the $s =0$ case given in \cite{ET}. 

It remains to consider the case where $1, |\xi| \ll |\beta|$. In this case, $\left[ \frac{\lb \xi \beta^{-1}\rb}{\lb \xi \rb} \right ]^{-s} \approx \lb \xi \rb^{s}$. Following the argument in \cite{ET} leads us to estimate 
\[ \left\|(\beta^2 + \xi^2)^{-\frac12 }\lb \xi \rb^{s}\right\|_{L^2_{|\xi| \ll |\beta|}} \approx \beta^{-1} \left\|\lb \xi \rb^{s}\right\|_{L^2_{|\xi| \ll |\beta|}}.\] It would suffice to show that this norm is bounded by $\beta^\frac{2s-1}{2}$. But since $-\frac12 < s < 0$, we have $\left\|\lb \xi \rb^{s}\right\|_{L^2_{|\xi| \ll |\beta|}} \approx \beta^{\frac{2s+1}{2}}$, which gives exactly the desired result. 

We turn to estimate the $L^2_{\xi}L^{1}_{\tau}$-norm. We start by estimating $\Vert \lb \xi \rb^{s}\hat{\eta S_{2}f} \Vert_{L^{2}_{\xi}L^{1}_{\tau}}$:
\begin{align*}
        \| \lb \xi \rb^{s} \hat{\eta S_{2}f} \|_{L^2_{\xi}L^{1}_{\tau}} &=\left \| \lb \xi \rb^{s} \int_{0}^{\infty} \hat{\eta}(\tau-\mu^{2})\hat{h}\left(\frac{\xi}{\mu}\right)\widehat{\chi f}(\mu^{2})d\mu \right \|_{L^2_{\xi}L^{1}_{\tau}}\\
        & \lesssim \left\| \lb \xi \rb^{s} \int_{0}^{\infty} \left | \hat{h}\left(\frac{\xi}{\mu}\right) \right | |\widehat{\chi f}(\mu^{2})| d\mu \right\|_{L^2_{\xi}} \\
        &= \left\| \lb \xi \rb^{s} \int_{0}^{\infty} |\hat{h}(\beta)| \left|\widehat{\chi f}\left(\frac{\xi^{2}}{\beta^{2}}\right)\right|\frac{|\xi|}{\beta^{2}} d\beta \right\|_{L^2_{\xi}} \\
        &\lesssim \int_{0}^{\infty} \frac{|\hat{h}(\beta)|}{|\beta|^{2}} \left \| \lb \xi \rb^{s}|\xi|\widehat{\chi f}\left(\frac{\xi^{2}}{\beta^{2}}\right)\right \|_{L^{2}_{\xi}} d\beta.
\end{align*}
Letting $\rho=\frac{\xi^{2}}{\beta^{2}}$, we have
\begin{align}
         \left \| \lb \xi \rb^{s}|\xi| \widehat{\chi f}\left(\frac{\xi^{2}}{\beta^{2}}\right)\right \|_{L^{2}_{\xi}} \approx |\beta|^{\frac{3}{2}} \left \{ \int \lb \beta \rho^{\frac{1}{2}} \rb^{2s}|\rho|^{\frac{1}{2}}|\hat{\chi f(\rho)}|^{2}d\rho\right \}^{\frac{1}{2}}. \label{eq:auxiliary}
\end{align}

If $s \geq 0$, we use the inequality $\lb \beta \rho^{\frac{1}{2}}\rb \lesssim \lb \beta \rb \lb \rho \rb^{\frac{1}{2}}$ to bound the right hand side of (\ref{eq:auxiliary}) by $|\beta|^{\frac{3}{2}}\lb \beta \rb^{s}\| \chi_{(0,\infty)} f\|_{H^{\frac{2s+1}{4}}}$.

Let $-\frac{1}{2}<s<0$. For $|\beta| \geq 1$, we use the inequality $\lb \beta \rho^{\frac{1}{2}}\rb^{2s} \lesssim \lb \rho \rb^{s}$. For $|\beta| < 1$, we use that $\lb \beta \rho^{\frac{1}{2}}\rb^{2s} \lesssim |\beta|^{2s}\lb \rho \rb^{s}$. 

Finally, the estimate $\Vert \lb \xi \rb^{s} \hat{\eta S_{1}f}\Vert_{L^2_{\xi}L^{1}_{\tau}} \lesssim \| \chi_{(0,\infty)}f\|_{H^{\frac{2s+1}{4}}}$ follows directly from Lemma \ref{estimate: Free Solutions} since $S_{1}f=S_{\R}\psi$, where $\hat{\psi}(\beta)=\beta \hat{f}(-\beta^{2})\chi_{(0,\infty)}(\beta)$.
\end{proof}

\begin{lem} \label{lemma: V_0^t(0,g) estimate}
Fix $b \in (\frac{1}{6},\frac{1}{2}]$. Let $k \geq \frac{1}{2}-3b$. Then for any $g$ satisfying $\chi_{(0,\infty)}g \in H^{\frac{k+1}{3}}(\mathbb{R})$, and $\frac{1}{2}<\alpha \leq \frac{k}{3}+1$, we have
    \begin{align*}
    \left \| \eta(t) W^t_0(0,g) \right \|_{\t Y^{k,b}_{\alpha}} \lesssim \| \chi_{(0,\infty)}g \|_{H^{\frac{k+1}{3}}}.
    \end{align*}
\end{lem}

\begin{proof}
The estimate $\| \eta(t) W^t_0(0,g)\|_{\t X^{k,b}} \lesssim \| \chi g \|_{H^\frac{k+1}3_t(\R)}$ is proved in {\cite[Lemma 4.5]{CT}}. We continue with the $D_{\alpha}$-norm.

Let $h(y)=e^{-i y\cos(\frac{\pi}{3})}e^{-y\sin(\frac{\pi}{3})}\rho(y)$. Then
\begin{align*}
    \widehat{\eta W_{0}^{t}}(0,g)(\xi,\tau)=\frac{3}{\pi}Re\int_{0}^{\infty} \hat{\eta}(\tau-\mu^{3})\hat{h}\left(\frac{\xi}{\mu}\right)\widehat{\chi g}(\mu^{3})\mu d\mu.
\end{align*}
Recall that by (\ref{eq: mean zero extension}), $\hat{h}$ is a Schwartz function such that $\hat{h}(0)=0$ which implies
\begin{align} \label{eq: Lipschitz}
    |\hat{h}(\eta)| \lesssim |\eta|.
\end{align}
Therefore for any $\mu>0$ we have
\begin{align*}
    \left\Vert\hat{h}\left(\frac{\xi}{\mu}\right)\right \Vert_{L^{2}_{|\xi|\leq 1}}=\left\{\mu\int_{|\eta|<\mu^{-1}}|\hat{h}(\eta)|^{2}d\eta\right\}^{\frac{1}{2}} 
    \lesssim \left\{\mu\int_{|\eta|<\mu^{-1}}|\eta|^{2}d\eta\right\}^{\frac{1}{2}}
    \lesssim \mu^{-1}.
\end{align*}
Also, note that 
\begin{align*}
    \left\Vert\hat{h}\left(\frac{\xi}{\mu}\right)\right \Vert_{L^{2}_{|\xi|\leq 1}} \lesssim \left\{\mu\int_{-\infty}^{\infty}|\hat{h}(\eta)|^{2}d\eta\right\}^{\frac{1}{2}} 
    \lesssim \mu^{\frac{1}{2}}.
\end{align*}
Therefore, for $\frac{1}{2}<\alpha \leq \frac{k}{3}+1$,
\begin{align*}
    \left\Vert \eta (t) W_{0}^{t}(0,g)\right\Vert_{D_{\alpha}} &\lesssim \left\Vert \int_{0}^{\infty} \lb \tau \rb^{\alpha}\hat{\eta}(\tau-\mu^{3})\hat{h}\left(\frac{\xi}{\mu}\right)\widehat{\chi g}(\mu^{3})\mu d\mu \right \Vert_{L^{2}_{\tau}L^{2}_{|\xi|\leq 1}}\\
        &\lesssim \left\Vert \int_{0}^{\infty} \lb\tau \rb^{\alpha}|\hat{\eta}(\tau-\mu^{3})|\left\Vert\hat{h}\left(\frac{\xi}{\mu}\right)\right \Vert_{L^{2}_{|\xi|\leq 1}}|\widehat{\chi g}(\mu^{3})||\mu| d\mu \right \Vert_{L^{2}_{\tau}}\\
        & \lesssim  \left\Vert \int_{0}^{1} \lb\tau \rb^{\alpha}|\hat{\eta}(\tau-\mu^{3})||\widehat{\chi g}(\mu^{3})||\mu|^{\frac{3}{2}} d\mu \right \Vert_{L^{2}_{\tau}} \\
        & \qquad +\left\Vert \int_{1}^{\infty} \lb\tau \rb^{\alpha}|\hat{\eta}(\tau-\mu^{3})||\widehat{\chi g}(\mu^{3})| d\mu \right \Vert_{L^{2}_{\tau}}=:\textit{I\,}+\textit{II\,}.
\end{align*}
To estimate $\textit{I\,}$, we recall that $\hat{\eta}$ is a Schwartz function, thus $|\hat{\eta}(\tau-\beta)| \lesssim \lb \tau - \beta \rb^{-1-\alpha}$:
\begin{align*}
        \textit{I\,} &\lesssim \left\Vert \int_{0}^{1} \lb\tau \rb^{\alpha}|\hat{\eta}(\tau-\beta)||\widehat{\chi g}(\beta)||\beta|^{-\frac{1}{6}} d\beta \right \Vert_{L^{2}_{\tau}}\\
        &\lesssim \left\Vert \int_{0}^{1} \lb\tau \rb^{\alpha}\lb\tau -\beta \rb^{-1-\alpha}|\widehat{\chi g}(\beta)||\beta|^{-\frac{1}{6}} d\beta \right \Vert_{L^{2}_{\tau}} \\
        &\lesssim \int_{0}^{1} \left\Vert \lb \tau \rb^{\alpha}\lb\tau -\beta \rb^{-1-\alpha}\right \Vert_{L^{2}_{\tau}}|\widehat{\chi g}(\beta)||\beta|^{-\frac{1}{6}} d\beta \\
        &\lesssim \left(\int_{0}^{1}|\beta|^{-\frac{1}{3}}d\beta\right)^{\frac{1}{2}} \left\Vert \widehat{\chi g} \right \Vert_{L^{2}([0,1])} \lesssim \left\Vert \chi_{(0,\infty)}g \right \Vert_{H^{\alpha-\frac{2}{3}}}. 
\end{align*}
We use Young's convolution inequality to estimate $\textit{II\,}$:
\begin{align*}
        \textit{II\,} &\lesssim \left\Vert \int_{1}^{\infty} \lb\tau \rb^{\alpha}|\hat{\eta}(\tau-\beta)||\widehat{\chi g}(\beta)||\beta|^{-\frac{2}{3}} d\beta \right \Vert_{L^{2}_{\tau}}\\
        &\lesssim \left\Vert \int_{1}^{\infty} \lb\tau-\beta \rb^{\alpha}|\hat{\eta}(\tau-\beta)||\widehat{\chi g}(\beta)|\lb\beta\rb^{\alpha-\frac{2}{3}} d\beta \right \Vert_{L^{2}_{\tau}} \\
        &\lesssim \left\Vert \lb\tau \rb^{\alpha}|\hat{\eta}(\tau)|\right \Vert_{L^{1}_{\tau}} \left\Vert |\widehat{\chi g}(\beta)|\lb\beta\rb^{\alpha-\frac{2}{3}} \right \Vert_{L^{2}_{\beta}} \\
        &\lesssim \left\Vert \chi_{(0,\infty)}g \right \Vert_{H^{\alpha-\frac{2}{3}}}.
\end{align*}

Finally, a similar argument as in the proof of Lemma \ref{lemma: S_0^t(0,f) estimate} shows $\| \lb \xi \rb^{k} \hat{\eta W_{0}^{t}}(0,g) \|_{L^2_{\xi}L^{1}_{\tau}} \lesssim \| \chi_{(0,\infty)}g\|_{H^{\frac{k+1}{3}}}$. We note that (\ref{eq: Lipschitz}) is essential for this estimate in the case $-\frac{3}{2} < k \leq -\frac{1}{2}$.
\end{proof}

\subsection{Nonlinear estimates}\label{section: nonlinear estimates}

\begin{lem}\label{lemma: inhomogeneous}
    For any $s, k \in \R$, $0 \leq b <\frac{1}{2}$ and $\alpha \leq 1-b$, we have 
    \begin{align*}
    \left \| \eta(t) \int_{0}^{t} S_{\R}^{t-t'}N(t')dt' \right \|_{Y^{s,b}} \lesssim \| N \|_{X^{s,-b}}
    \end{align*}
    and
    \begin{align*}
    \left \| \eta(t) \int_{0}^{t} W_{\R}^{t-t'}N(t') dt' \right \|_{\t Y^{k,b}_{\alpha}} \lesssim \| N \|_{\t X^{k,-b}}.
    \end{align*}
\end{lem}

\begin{proof}
We only prove the second inequality. First, it is proved in \cite[Lemma 5.4]{CK} that
\begin{align*}
    \left \| \eta(t) \int_{0}^{t} W_{\R}^{t-t'}N(t') dt' \right \|_{\t X^{k,b}_{\alpha}} \lesssim \| N \|_{\t X^{k,-b}}.
\end{align*}
The following estimate is standard (e.g., see the proof of \cite[Lemma 3.16]{ET1}):
\begin{align*}
    \left \| \lb \xi \rb^{k}\mathcal{F}\left [\eta(t) \int_{0}^{t} W_{\R}^{t-t'}N(t')dt' \right ] \right \|_{L^{2}_{\xi}L^{1}_{\tau}} \lesssim \|\lb \xi \rb^{k}\lb \tau-\xi^{3} \rb^{-1}\hat{N}(\xi, \tau) \|_{L^2_{\xi}L^{1}_{\tau}}.
\end{align*}
However, since 
\begin{align*}
        \| \lb \xi \rb^{k}\lb \tau -\xi^{3} \rb^{-1} \hat{N} \|^{2}_{L^{2}_{\xi}L^{1}_{\tau}} 
        &=\int\left ( \int \lb \xi \rb^{k}\lb \tau -\xi^{3} \rb^{-1} |\hat{N}| d\tau\right)^{2} d\xi \\
        &\leq \int\left (\int \lb \xi \rb^{2k}\lb \tau -\xi^{3} \rb^{-2b} |\hat{N}|^{2} d\tau\right) \left (\int \frac{d\tau}{\lb \tau -\xi^{3} \rb^{2(1-b)}}\right) d\xi \\
        &\approx \iint \lb \xi \rb^{2k}\lb \tau -\xi^{3} \rb^{-2b} |\hat{N}|^{2} d\tau d\xi \\
        &=\| N \|^{2}_{\t X^{k,-b}},
\end{align*}
the desired bound follows.
\end{proof}

\begin{lem}[{\cite[Proposition 3.7]{ET2}}] \label{lemma: NLS Correction}
	Fix $b \in (0,\frac12)$. Let $R$ be the set given by $R = \bigl\{ (\xi,\tau) \; : \; |\tau| \geq 1  \quad \text{and} \quad |\tau | \gg |\xi|^2\bigr\}$. Then we have
	\begin{multline*}
	\left\| \eta(t) \int_{0}^{t} S^{t-t'}_{\R}N(t')d t' \right\|_{C_x H^\frac{2s+1}{4}_t( \R \times \R)} \\ \lesssim
	\begin{cases}
	\| N\|_{X^{s,-b}} &\; 0 \leq s \leq \frac12, \\
	\| N \|_{X^{s,-b}} + \left\| \int \chi_R(\xi, \tau) \lb \tau + \xi^2 \rb^{\frac{2s-3}4} |\hat{N}(\xi,\tau)| d \xi \right\|_{L^2_\tau} & \qquad  s > \frac12. 
	\end{cases}
	\end{multline*}
\end{lem}

\begin{lem}[{\cite[Proposition 4.8]{CT}}] \label{lemma: KdV Correction}
Fix $b \in (0,\frac12)$. Let $R$ be the set given by $R = \bigl\{ (\xi,\tau) \; : \; |\tau| \geq 1 \quad \text{and} \quad |\tau | \gg |\xi|^3\bigr\}$. Then we have
\begin{multline*}
  \left\| \eta(t) \int_0^t W^{t-t'}_{\R}N(t')d t'  \right\|_{C_x H^\frac{k+1}{3}_t (\R \times \R)} 
\\ \lesssim
\begin{cases}
\| N \|_{\t X^{k,-b}} & -1 \leq k \leq 2 - 3b, \\
\| N \|_{\t X^{k,-b}} + \left\| \int \chi_R(\xi, \tau) \lb \tau - \xi^3 \rb^{\frac{k-2}{3}} | \hat{N}(\xi,\tau)| d \xi \right\|_{L^2_\tau} &\qquad \;\;k > 2-3b.
\end{cases}
\end{multline*} 
\end{lem}

We need the following estimates to apply Lemmas \ref{lemma: inhomogeneous}, \ref{lemma: NLS Correction} and \ref{lemma: KdV Correction} to the nonlinear terms in (\ref{eq: SKdV}). Proofs are in Section \ref{section: proofs of estimates}.
\begin{prop} \label{estimate:uv}
Let $s<k+2$, $k > -1$ and $s + k > -1$. Then for any 
\begin{align*}
\begin{array}{lr}
        0\leq a <\frac12 \min\left( k-s+2,k+1, 1\right), & \text{if } s \geq 0,\\
        0\leq a <\frac12\min\left( s+k+\frac32,k+1, 1\right), & \text{if } s<0,
\end{array}
\end{align*}
there exists $\epsilon >0$ such that for $\frac{1}{2}-\epsilon < b <\frac{1}{2}$, we have 
\[ \| uv \|_{X^{s+a,-b}} \lesssim \| u \|_{X^{s,b}} \|v\|_{\t X^{k,b}}.\]
\end{prop}

\begin{prop} \label{estimate:|u|^2u}
Let $s \geq 0$. For $0 \leq a \leq \min( 2s, \frac12)$, there exists $\epsilon >0$ such that for $\frac{1}{2}-\epsilon < b <\frac{1}{2}$, we have
\[ \| u_1 \overline{u_2} u_3 \|_{X^{s+a, -b}} \lesssim \| u_1 \|_{X^{s,b}}\| u_2\|_{X^{s,b}} \| u_3\|_{X^{s,b}}.\]	
\end{prop}

\begin{prop} \label{estimate:vv_x}
   For $k>-\frac{3}{4}$ and $\alpha>\frac{1}{2}$, there exists $\epsilon >0$ such that for any $\frac{1}{2}-\epsilon<b<\frac{1}{2}$, we have
    \begin{align*}
    \| \partial_{x}(v_{1}v_{2}) \|_{\t X^{k,-b}} \lesssim \| v_{1} \|_{\t X^{k,b}_{\alpha}}\| v_{2} \|_{\t X^{k,b}_{\alpha}}.
    \end{align*}
\end{prop}

\begin{prop} \label{estimate:|u|^2_x}
Let $k<\min(4s,s+1)$. For $s \geq 0$ and $0 \leq a <\min(4s-k,s-k+1)$ or $-\frac14 <s<0$ and $0 \leq a <\min(4s-k,-k)$, there exists $\epsilon >0$ such that for $\frac{1}{2}-\epsilon<b<\frac{1}{2}$, we have
\[ \| \partial_{x}\left(u_1 \overline{u_2}\right) \|_{\t X^{k+a,-b}} \lesssim \| u_1 \|_{X^{s,b}} \|u_2\|_{X^{s,b}}. \]
\end{prop}

\begin{prop} \label{estimate:uv correction}
Let $N=uv$. For $\frac{1}{2}<s<\min(\frac{2}{3}k+\frac{13}{6},\frac{5}{2})$, $k>-1$ and $0\leq a <\min(\frac{2}{3}k-s+\frac{13}{6},\frac{5}{2}-s)$, there exists $\epsilon > 0$ such that for $\frac{1}{2}-\epsilon<b<\frac{1}{2}$, we have 
\begin{align*} 
	\left\Vert \int{\chi_{R}\lb \tau+\xi^{2}\rb}^{\frac{2(s+a)-3}{4}}|\hat{N}(\xi,\tau)|d\xi\right\Vert_{L^{2}_{\tau}} \lesssim \left\Vert u \right \Vert_{X^{s,b}} \left\Vert v\right\Vert_{\t X^{k,b}},
\end{align*}
where $R=\{(\xi,\tau):|\tau|\gg 1  \text{ and } |\tau| \gg |\xi|^{2} \}$.
\end{prop}

\begin{prop}[{\cite[Proposition 3.6]{ET}}] \label{estimate:|u|^2u correction}
Let $N = u_1\overline{u_2} u_3$. For $\frac12 < s <\frac52$ and $0 \leq a < \min(2s, \frac{5}{2}-s, \frac12)$, there exists $\epsilon >0$ such that for $\frac{1}{2}-\epsilon < b <\frac{1}{2}$, we have 
\[\left\| \int \lb \tau + \xi^2 \rb^{\frac{2(s+a)-3}4} |\hat{N}(\xi,\tau)| d \xi \right\|_{L^2_\tau} \lesssim \| u_1\|_{X^{s,b}}\| u_2\|_{X^{s,b}}\| u_3\|_{X^{s,b}}.\]
\end{prop}

\begin{prop} \label{estimate:vv_x correction}
Let $N(x,t) = \partial_{x}(v_{1}v_{2})$. For $\frac{1}{2}<k<\frac{7}{2}$ and $0 \leq a < \min(\frac{7}{2}-k,\frac12)$, there exists $\epsilon >0$ such that for $\frac{1}{2}-\epsilon < b <\frac{1}{2}$, we have  
    \begin{align*}
    \left\| \int \chi_R(\xi, \tau) \lb \tau - \xi^3 \rb^{\frac{s+a-2}{3}} | \hat{N}(\xi,\tau)| d \xi \right\|_{L^2_\tau} \lesssim \| v_{1} \|_{\t X^{k,b}}\| v_{2} \|_{\t X^{k,b}},
    \end{align*}
    where $R = \bigl\{ (\xi,\tau) \; : \; |\tau| \geq 1 \text{ and } |\tau | \gg |\xi|^3\bigr\}$.
\end{prop}

\begin{prop} \label{estimate:|u|^2_x correction}
  Let $N(x,t) = \partial_{x}(u_{1}\overline{u_2})$, and define the set $R$ as in the previous proposition: $R = \bigl\{ (\xi,\tau) \, : \, |\tau| \geq 1 \text{ and } |\tau | \gg |\xi|^3\bigr\}$. Then for $s\geq0$, $k<\min(2s+\frac12,s+\frac{3}{4},\frac{7}{2})$ and $0 \leq a < \min(s-k+1,\frac{7}{2}-k,\frac12)$, there exists $\epsilon >0$ such that for $\frac{1}{2}-\epsilon < b <\frac{1}{2}$, we have
 \[ \left\| \int \chi_R(\xi, \tau) \lb \tau - \xi^3 \rb^{\frac{k +a-2}{3}} | \hat{N}(\xi,\tau)| d \xi \right\|_{L^2_{\tau}} \lesssim  \| u_{1}\|_{X^{s,b}} \| u_{2}\|_{X^{s,b}}. \]
\end{prop}

\begin{rmk} \label{rem: Trivial smoothings}
The proofs of estimates in Propositions \ref{estimate:|u|^2u} and \ref{estimate:vv_x} trivially imply
\begin{multline} \label{eq: trivial smoothing 1}
    \| u_1 \overline{u_2} u_3 \|_{X^{s+a, -b}} \lesssim \| u_1 \|_{X^{s+a,b}}\| u_2\|_{X^{s,b}} \| u_3\|_{X^{s,b}}\\
    +\| u_1 \|_{X^{s,b}}\| u_2\|_{X^{s+a,b}} \| u_3\|_{X^{s,b}}+ \| u_1 \|_{X^{s,b}}\| u_2\|_{X^{s,b}} \| u_3\|_{X^{s+a,b}}
\end{multline}
for all $s \geq 0$ and $a \geq 0$, and
\begin{align} \label{eq: trivial smoothing 2}
    \| \partial_{x}(v_{1}v_{2}) \|_{\t X^{k+a,-b}} \lesssim \| v_{1} \|_{\t X^{k+a,b}_{\alpha}}\| v_{2} \|_{\t X^{k,b}_{\alpha}}+\| v_{1} \|_{\t X^{k,b}_{\alpha}}\| v_{2} \|_{\t X^{k+a,b}_{\alpha}},
\end{align}
for all $k>-\frac34$ and $a \geq 0$. For example, to prove (\ref{eq: trivial smoothing 2}) it suffices to demonstrate
    \begin{multline*}
        \iint_{\begin{subarray}{l}\\ \xi=\xi_{1}+\xi_{2}\\ \tau=\tau_{1}+\tau_{2} \end
{subarray}}\frac{\lb\xi\rb^{s+a}|\xi||h(\xi,\tau)|}{\lb \tau - \xi^{3}\rb^{b}}\frac{|f(\xi_{1},\tau_{1})|}{\lb \xi_{1} \rb^{s}\beta(\xi_{1},\tau_{1})}\frac{|g(\xi_{2},\tau_{2})|}{\lb \xi_{2} \rb^{s}\beta(\xi_{2},\tau_{2})} \\
 \lesssim  \Vert \lb \xi \rb^{a}f \Vert_{L^{2}_{\xi,\tau}}\Vert g \Vert_{L^{2}_{\xi,\tau}}\Vert h \Vert_{L^{2}_{\xi,\tau}}+\Vert f \Vert_{L^{2}_{\xi,\tau}}\Vert \lb \xi \rb^{a}g \Vert_{L^{2}_{\xi,\tau}}\Vert h \Vert_{L^{2}_{\xi,\tau}},
    \end{multline*}
 where $\beta(\xi,\tau)=\lb \tau - \xi^{3}\rb^{b}+\chi_{|\xi|\leq1}\lb \tau \rb^{\alpha}$.
However, since
\begin{align*}
    \lb \xi \rb^{a} = \lb \xi_1+\xi_2 \rb^{a}\lesssim \lb \xi_1\rb^{a}+\lb \xi_2\rb^{a},
\end{align*}
the desired bound directly follows from the $a=0$ case, which has been already established in \cite[Lemma 5.10 (a)]{Holm1}. 
\end{rmk}

\section{Proof of the main results} \label{section: Main results}

\subsection{A contraction argument} \label{section: existence}
In this section, we show the local existence of solutions to (\ref{eq: SKdVintegral}) by a contraction argument. Define the map $\Gamma=(\Gamma_{1},\Gamma_{2})$ by
\begin{align*}
        &\Gamma_{1}(u,v)=\eta(t)S_{\R}^{t}u_{0e}-i\eta(t)\int_{0}^{t}S_{\R}^{t-t'}F(u,v)dt'+\eta(t)S_{0}^{t}(0,f-p-q), \\
        &\Gamma_{2}(u,v)=\eta(t)W_{\R}^{t}v_{0e}+\eta(t)\int_{0}^{t}W_{\R}^{t-t'}G(u,v)dt'+\eta(t)W_{0}^{t}(0,g-\t p-\t q),
\end{align*}
where 
\begin{multline*}
    F(u,v)=\eta(t/T)(\alpha uv+\beta|u|^{2}u), \quad p(t)=\eta(t)D_{0}(S_{\R}^{t}u_{0e}), \\
    q(t)=\eta(t)D_{0}\left ( \int_{0}^{t}S_{\R}^{t-t'}F(u,v)dt'\right ), \quad
    G(u,v)=\eta(t/T)(\gamma(|u|^{2})_{x}-vv_{x}), \\
    \t p(t)=\eta(t)D_{0}(W_{\R}^{t}v_{0e}), \quad \t q(t)=\eta(t)D_{0}\left ( \int_{0}^{t}W_{\R}^{t-t'}G(u,v)dt'\right ).
\end{multline*}
Let $(s,k)$ be an admissible pair. We start by estimating the $Y^{s,b}$-norm of $\Gamma_{1}(u,v)$. Pick any $b'$ such that $\frac{1}{2}-\epsilon<b'<b<\frac{1}{2}$, where $\epsilon$ is a constant provided in Propositions \ref{estimate:uv}, \ref{estimate:|u|^2u}, \ref{estimate:uv correction} and \ref{estimate:|u|^2u correction}. Combining Lemma \ref{lemma: inhomogeneous}, Propositions \ref{estimate:uv} and \ref{estimate:|u|^2u}, we have
\begin{multline*}
    \left \| \eta(t)\int_{0}^{t}S_{\R}^{t-t'}F(u,v)dt' \right \|_{Y^{s,b} } \lesssim \|F(u,v)\|_{X^{s,-b}}
    \lesssim T^{b-b'} \|F(u,v)\|_{X^{s,-b'}} \\
    \lesssim T^{b-b'} (\|u\|_{Y^{s,b'} }\|v\|_{\t Y^{k,b'}_{\alpha} }+\|u\|_{Y^{s,b'}}^{3}) 
    \lesssim T^{b-b'} (\|u\|_{Y^{s,b} }\|v\|_{\t Y^{k,b}_{\alpha} }+\|u\|_{Y^{s,b}}^{3}).
\end{multline*}
Moreover, by Lemmas \ref{lemma: S_0^t(0,f) estimate} and \ref{lemma: extension}, we have
\begin{multline*}
    \| \eta(t)S_{0}^{t}(0,f-p-q) \|_{Y^{s,b}} \lesssim \|(f-p-q)\chi_{(0,\infty)}\|_{H^{\frac{2s+1}{4}}_{t}(\R)} \\
    \lesssim \|f-p \|_{H^{\frac{2s+1}{4}}_{t}(\R^+)}+\|q \|_{H^{\frac{2s+1}{4}}_{t}(\R^+)} \lesssim \|f \|_{H^{\frac{2s+1}{4}}_{t}(\R^+)}+\|p \|_{H^{\frac{2s+1}{4}}_{t}(\R)}+\|q \|_{H^{\frac{2s+1}{4}}_{t}(\R)}.
\end{multline*}
By the Kato smoothing Lemma \ref{lemma: Kato Smoothing},
\begin{align*}
    \|p \|_{H^{\frac{2s+1}{4}}_{t}(\R)} \lesssim \|\eta(t)S_{\R}^{t}u_{0e} \|_{C_{x}H^{\frac{2s+1}{4}}_{t}(\R \times \R)} \lesssim \| u_{0} \|_{H^{s}(\R^{+})}.
\end{align*}
By Propositions \ref{estimate:uv correction} and \ref{estimate:|u|^2u correction},
\begin{multline} \label{eq: correction scaling}
    \left\| \int \chi_R(\xi, \tau) \lb \tau + \xi^2 \rb^{\frac{2s-3}4} |\hat{F(u,v)}(\xi,\tau)| d \xi \right\|_{L^2_\tau} \\
    \lesssim \|\eta(t/T)u\|_{X^{s,b'}}\|v\|_{\t X^{k,b'}}+\|\eta(t/T)u\|_{X^{s,b'}}\|u\|_{X^{s,b'}}^2 \\
    \lesssim T^{b-b'}(\|u\|_{Y^{s,b} }\|v\|_{\t Y^{k,b}_{\alpha} }+\|u\|_{Y^{s,b} }^{3}).
\end{multline}
Using Lemma \ref{lemma: NLS Correction} and (\ref{eq: correction scaling}), we get
\begin{multline*}
   \|q\|_{H^{\frac{2s+1}{4}}_{t}(\R)} \lesssim \left \|\eta(t)\int_{0}^{t}S_{\R}^{t-t'}F(u,v)dt'\right \|_{C_{x}H^{\frac{2s+1}{4}}_{t}(\R \times \R))} \\
   \lesssim \begin{cases}
	\| F(u,v)\|_{X^{s,-b}} &\; 0 \leq s \leq \frac12, \\
	\| F(u,v) \|_{X^{s,-b}} + \left\| \int \chi_R(\xi, \tau) \lb \tau + \xi^2 \rb^{\frac{2s-3}4} |\hat{F(u,v)}(\xi,\tau)| d \xi \right\|_{L^2_\tau} & \qquad  s > \frac12. 
	\end{cases} \\
    \lesssim T^{b-b'}(\|u\|_{Y^{s,b} }\|v\|_{\t Y^{k,b}_{\alpha} }+\|u\|_{Y^{s,b} }^{3}).
\end{multline*}

Therefore, we have shown
\begin{align} \label{eq: contraction1}
    \|\Gamma_{1}(u,v) \|_{Y^{s,b}} 
    \lesssim \| u_{0} \|_{H^{s}(\R^{+})}+\|f \|_{H^{\frac{2s+1}{4}}(\R^{+})}+T^{b- b'}(\|u\|_{Y^{s,b} }\|v\|_{\t Y^{k,b}_{\alpha} }+\|u\|_{Y^{s,b} }^{3}).
\end{align}
We can similarly show that 
\begin{align} \label{eq: contraction2}
    \|\Gamma_{2}(u,v) \|_{\t Y^{k,b}_{\alpha}}
    \lesssim \| v_{0} \|_{H^{k}(\R^{+})}+\|g \|_{H^{\frac{k+1}{3}}(\R^{+})}
    +T^{b- b'}(\|u\|_{Y^{s,b} }^{2}+\|v\|_{\t Y^{k,b}_{\alpha} }^{2}).
\end{align}
We can also estimate the difference $\|\Gamma_{1}(u,v)-\Gamma_{1}(\t u,\t v) \|_{Y^{s,b}}+ \|\Gamma_{2}(u,v)-\Gamma_{2}(\t u,\t v)\|_{\t Y^{k,b}_{\alpha}}$ in a similar way. Therefore, by a standard fixed point argument we can ensure that for some small enough $T>0$, the system (\ref{eq: SKdVintegral}) has a unique solution $(u,v) \in  Y^{s,b}([0,T]) \times \t Y^{k,b}_{\alpha}([0,T])$. That $(u,v) \in C_{x}H^{\frac{2s+1}{4}}_{t}(\R \times [0,T]) \times C_{x}H^{\frac{k+1}{3}}_{t}(\R \times [0,T])$ follows from Lemmas \ref{lemma: Kato Smoothing}, \ref{linHs}, \ref{lemma: NLS Correction} and \ref{lemma: KdV Correction}.

\subsection{Persistence of regularity} \label{section: persistence}
In this section we establish the following persistence of regularity result. The proof relies on the smoothing estimates in section \ref{section: nonlinear estimates}.
\begin{prop} \label{prop: persistence}
    Let $(s_{0},k_{0})$ and $(s_{1}, k_{1})$ be admissible pairs such that $s_{0} \leq s_{1}$ and $k_{0} \leq k_{1}$. Let $(u_0,f,v_0,g) \in \mathcal{H}^{s_1,k_1}$. Suppose that the corresponding solution $(u,v)$ solves (\ref{eq: SKdVintegral}) in $Y^{s_{0},b}([0,T^{*}])\times \t Y^{k_{0},b}_{\alpha}([0,T^{*}])$. Then we have
    \begin{align*}
        (u,v) \in Y^{s_{1},b}([0,T^{*}])\times \t Y^{k_{1},b}_{\alpha}([0,T^{*}]).
    \end{align*}
\end{prop}
\begin{proof}
    We first assume that $(s_{1}, k_{1})$ satisfies $s_{0} \leq s_{1} <s_{0}+a$ and $k_{0} \leq k_{1} <k_{0}+a$, where $a$ is a positive smoothing index provided in Propositions \ref{estimate:uv correction}, \ref{estimate:|u|^2u correction}, \ref{estimate:vv_x correction} and \ref{estimate:|u|^2_x correction}. Let $0<T\leq T^{*}$. Taking the $Y^{s_{1},b}([0,T])\times Y^{k_{1},b}_{\alpha}([0,T])$-norm of both sides of (\ref{eq: SKdVintegral}) and then applying the smoothing estimates (with estimates in Remark \ref{rem: Trivial smoothings}), we have for some $\epsilon >0$
\begin{multline*}
     \|u\|_{Y^{s_{1},b}([0,T])} \lesssim \|u_{0e}\|_{H^{s_{1}}}+\|f\|_{H_{\R^{+}}^{\frac{2s_{1}+1}{4}}} \\
     +T^{\epsilon}(\|u\|_{Y^{s_{0},b}([0,T])}\|v\|_{\t Y^{k_{0},b}_{\alpha}([0,T])}
     +\|u\|_{Y^{s_{0},b}([0,T])}^{2}\|u\|_{Y^{s_{1},b}([0,T])})
\end{multline*}
and
\begin{align*}
     \|v\|_{\t Y^{k_{1},b}_{\alpha}([0,T])} \lesssim \|v_{0e}\|_{H^{k_{1}}}+\|g\|_{H_{\R^{+}}^{\frac{k_{1}+1}{3}}}+T^{\epsilon}(\|v\|_{\t Y^{k_{0},b}_{\alpha}([0,T])}\|v\|_{\t Y^{k_{1},b}_{\alpha}([0,T])}
     +\|u\|_{Y^{s_{0},b}([0,T])}^{2}).
\end{align*}
Hence we have
\begin{multline*}
    \|u\|_{Y^{s_{1},b}([0,T])}+\|v\|_{\t Y^{k_{1},b}_{\alpha}([0,T])} 
    \lesssim  \|u_{0e}\|_{H^{s_{1}}}+\|f\|_{H_{\R^{+}}^{\frac{2s_{1}+1}{4}}}+\|v_{0e}\|_{H^{k_{1}}}+\|g\|_{H_{\R^{+}}^{\frac{k_{1}+1}{3}}} \\
    +T^{\epsilon}(\|u\|_{Y^{s_{0},b}([0,T^{*}])}+\|u\|_{Y^{s_{0},b}([0,T^{*}])}^{2}+\|v\|_{\t Y^{k_{0},b}_{\alpha}([0,T^{*}])}) 
    (\|u\|_{Y^{s_{1},b}([0,T])}+\|v\|_{\t Y^{k_{1},b}_{\alpha}([0,T])}).
\end{multline*}
Therefore, as long as $\|u\|_{Y^{s_{1},b}([0,T])}+\|v\|_{\t Y^{k_{1},b}_{\alpha}([0,T])}$ is finite and 
\begin{align*}
    T^{\epsilon}(\|u\|_{Y^{s_{0},b}([0,T^{*}])}+\|u\|_{Y^{s_{0},b}([0,T^{*}])}^{2}+\|v\|_{\t Y^{k_{0},b}_{\alpha}([0,T^{*}])}) \ll 1,
\end{align*}
we have 
\begin{align} \label{eq: persistence1}
    \|u\|_{Y^{s_{1},b}([0,T])}+\|v\|_{\t Y^{k_{1},b}_{\alpha}([0,T])} \lesssim  \|u_{0e}\|_{H^{s_{1}}}+\|f\|_{H_{\R^{+}}^{\frac{2s_{1}+1}{4}}}+\|v_{0e}\|_{H^{k_{1}}}+\|g\|_{H_{\R^{+}}^{\frac{k_{1}+1}{3}}}.
\end{align}
By a standard blow-up alternative argument, we can pick $T>0$ only depending on $\|u\|_{Y^{s_{0},b}([0,T^{*}])}$ and  $\|v\|_{Y^{k_{0},b}_{\alpha}([0,T^{*}])}$ such that (\ref{eq: persistence1}) holds. 

Pick a fixed small number $0<\delta <T/2$. Noticing that 
\begin{align*}
\| u(T-\delta) \|_{H^{s_{1}}}\leq \| u \|_{Y^{s_{1},b}([0,T])} \quad \text{and} \quad \| v(T-\delta) \|_{H^{k_{1}}} \leq \| v \|_{\t Y^{k_{1},b}_{\alpha}([0,T])},
\end{align*}
we can iterate the above argument on $[T-\delta,2T-\delta]$, $[2T-2\delta,3T-2\delta]$, $\dots$ and then apply Lemma \ref{lemma: concatenation} to obtain $\|u\|_{Y^{s_{1},b}([0,T^{*}])}+\|v\|_{\t Y^{k_{1},b}_{\alpha}([0,T^{*}])} <\infty$.

Finally, for a general admissible $(s_{0},k_{0})$ such that $s_{0} \leq s_{1}$ and $k_{0} \leq k_{1}$, induction on the Sobolev index establishes the bound 
\begin{align*}
    \|u\|_{Y^{s_{1},b}([0,T^{*}])}+\|v\|_{\t Y^{k_{1},b}_{\alpha}([0,T^{*}])} < \infty.
\end{align*}
\end{proof} 

\begin{rmk} For the real-line problem, the persistence of regularity would readily follow from the trivial one-sided smoothing estimate as in Remark \ref{rem: Trivial smoothings}. For the half-line problem it seems that additional correction terms
\begin{align*}
    \left\| \int \chi_R(\xi, \tau) \lb \tau + \xi^2 \rb^{\frac{2s-3}4} |\hat{N}(\xi,\tau)| d \xi \right\|_{L^2_\tau}
\end{align*}
and
\begin{align*}
\left\| \int \chi_R(\xi, \tau) \lb \tau - \xi^3 \rb^{\frac{k-2}{3}} | \hat{N}(\xi,\tau)| d \xi \right\|_{L^2_\tau}
\end{align*}
arising from Proposition \ref{lemma: NLS Correction} and Proposition \ref{lemma: KdV Correction} do not have similar properties. Therefore, we use the full smoothing estimates of these correction terms. 
\end{rmk}

\subsection{Uniqueness of solutions} \label{section: uniqueness}
Since the solutions we constructed in Section \ref{section: existence} were obtained as fixed points of operators that depend on the extensions of the initial data $(u_0,v_0)$, it is not clear whether or not these solutions depend on the corresponding extensions. That is, to ensure the uniqueness of the solutions to equation (\ref{eq: SKdV}), it is necessary to show that these solutions do not depend on the extensions of the initial data. If $s>\frac{3}{2}$ and $k>\frac{5}{2}$, the uniqueness follows from the energy estimate, which is proved in Section \ref{section: high regularity uniqueness}. In this section we establish the uniqueness of low-regularity solutions. For this, it suffices to establish the following lemma:
\begin{lem} \label{lemma: approximation}
Fix an admissible pair $(s_0,k_0)$ such that $s_0>\frac32,\: k_0>\frac52$. Let $(s,k)$ be an admissible pair such that $s<s_0$ and $k<k_0$. For $(u_0,f,v_0,g) \in \mathcal{H}^{s,k}$, fix $(u_{0e},v_{0e}) \in H^s(\R) \times H^k(\R)$ which extends $(u_0,v_0)$. Let $(u,v) \in Y^{s}([0,T^{*}]) \times \t Y^{k}_{\alpha}([0,T^{*}])$ be a local solution to (\ref{eq: SKdVintegral}) corresponding to $(u_{0e},f,v_{0e},g)$. Then for any $T<T^{*}$, there exists a sequence $\{(u_{0n},f_n,v_{0n},g_n)\} \subset \mathcal{H}^{s_0,k_0}$ satisfying the following conditions: \\
    (i) Let $(u_n,v_n)$ be a unique local solution to (\ref{eq: SKdVintegral}) corresponding to the initial and boundary data $(u_{0n},f_n,v_{0n},g_n)$. Then $(u_n,v_n)$ extends to the interval $[0,T]$, that is, 
    \begin{align*}
         (u_n,v_n) \in C_{t}H^{s_0}_x([0,T]\times \R^+) \times C_{t}H^{k_0}_x([0,T]\times \R^+). 
    \end{align*}
    (ii) We have
    \begin{align*}
         \lim_{n \to \infty}\left [ \|u-u_n\|_{C_{t}H^{s}_x([0,T]\times \R^+)}+\|v-v_n\|_{C_{t}H^{k}_x([0,T]\times \R^+)} \right ]=0.
    \end{align*}
\end{lem}
Again, the uniqueness for $(u_n,v_n)$ follows from the energy estimates in Section \ref{section: high regularity uniqueness}. 

\begin{proof}
  Take a sequence $\{(u_{0n},f_n,v_{0n},g_n)\} \subset \mathcal{H}^{s_0,k_0}$ converging to $(u_0,f,v_0,g)$ in $ \mathcal{H}^{s,k}$. Using Lemma \ref{lemma: extension2}, we can choose a sequence $\{(u_{0en},v_{0en})\} \subset H^{2}(\R) \times H^{3}(\R)$ satisfying
\begin{align} \label{eq: extension}
    \|u_{0e}- u_{0en}\|_{H^{s}(\R)}+ \|v_{0e} - v_{0en}\|_{H^{k}(\R)} \lesssim  \|u_0 - u_{0n}\|_{H^{s}(\R^+)}+ \|v_0 - v_{0n}\|_{H^{k}(\R^+)},
\end{align}
which extends $\{ (u_{0n},v_{0n}) \}$ for each $n$. Let $(u_n,v_n)$ be a $H^{s}(\R^+) \times H^{k}(\R^+)$-solution to (\ref{eq: SKdVintegral}) corresponding to $(u_{0en},f_n,v_{0en},g_n)$, and let $[0,T_n]$ be the $H^{s}(\R^+) \times H^{k}(\R^+)$-level maximal existence interval of $(u_n,v_n)$. By Proposition \ref{prop: persistence}, $(u_n,v_n)$ remain in $H^{2}(\R^+) \times H^{3}(\R^+)$ on the time interval $[0,T_n]$. 
Take any $0\leq T<T^{*}$. Then we have $T<T_n$ for all large $n$. Using the local Lipschitzness of the solution map and (\ref{eq: extension}), we have
\begin{multline*}
   \|u-u_n\|_{C_{t}H^{s}_x([0,T]\times \R^+)}+\|v-v_n\|_{C_{t}H^{k}_x([0,T]\times \R^+)} 
   \lesssim \|u-u_n\|_{Y^{s,b}([0,T])}+\| v-v_n\|_{\t Y^{k,b}_{\alpha}([0,T])} \\ \lesssim  \|u_{0e}- u_{0en}\|_{H^{s}(\R)}+ \|v_{0e} - v_{0en}\|_{H^{k}(\R)}+\|f-f_n\|_{H^{\frac{2s+1}{4}}(\R^+)}+\|g-g_n\|_{H^{\frac{k+1}{3}}(\R^+)} \\
    \lesssim \|u_0 - u_{0n}\|_{H^{s}(\R^+)}+ \|v_0 - v_{0n}\|_{H^{k}(\R^+)}+\|f-f_n\|_{H^{\frac{2s+1}{4}}(\R^+)}+\|g-g_n\|_{H^{\frac{k+1}{3}}(\R^+)}.
\end{multline*} 
\end{proof}

We establish the uniqueness using an approximation argument. Fix $(u_0,v_0) \in H^{s}(\R^+) \times H^{k}(\R^+)$ and let $(u_{0e},v_{0e}), \:(\t u_{0e},\t v_{0e}) \in H^{s}(\R) \times H^{k}(\R)$ be two different extensions of $(u_0,v_0)$. Let $(u,v),\: (\t u, \t v)\in Y^{s}([0,T^{*}]) \times \t Y^{k}_{\alpha}([0,T^{*}])$ be local solutions to (\ref{eq: SKdVintegral}) corresponding to $(u_{0e},v_{0e})$ and $(\t u_{0e},\t v_{0e})$ respectively. Pick an admissible pair $(s_0,k_0)$ such that $s_0>\frac32, \: k_0>\frac52$. By Lemma \ref{lemma: approximation}, for any time interval $[0,T] \subset [0,T^{*})$, we can choose a sequence of solutions $(u_n,v_n) \in C_{t}H^{s_0}_x([0,T]\times \R^+) \times C_{t}H^{k_0}_x([0,T]\times \R^+)$ that approximates $(u,v)$ and $(\t u, \t v)$ simultaneously. This implies the desired uniqueness $(u,v)=(\t u, \t v)$.

\begin{lem}[{\cite[Lemma 4.1]{EGT}}]\label{lemma: extension2}
Fix $s \in \R$ and $k \geq$ s. Let $g \in H^s(\R^+)$, $f \in H^k(\R^+)$ and let $g_e$ be an $H^s$ extension of $g$ to $\R$. Then there is an $H^k$ extension $f_e$ of $f$ to $\R$ so that
\begin{align*}
    \|g_e-f_e\|_{H^s(\R)} \lesssim \|g-f\|_{H^{s}(\R^+)}.
\end{align*}
\end{lem}
Note that this lemma is stated in \cite{EGT} for $s \geq 0$, but it holds for any $s \in \R$ by the same proof.

\subsection{Nonlinear smoothing}
Let $[0,T]$ be the local existence interval for (\ref{eq: SKdV}), given the initial and boundary data $(u_0,f,v_0,g) \in \mathcal{H}^{s,k}$. Then we have for $0\leq t \leq T$,
\begin{align*}
    u-S_{0}^{t}(u_{0},f)=\eta(t)\int_{0}^{t}S_{\R}^{t-t'}F(u,v)dt'-\eta(t)S_{0}^{t}(0,q),
\end{align*}
where 
\begin{align*}
    F(u,v)=\eta(t/T)(\alpha uv+\beta|u|^{2}u)  \quad \text{and} \quad q(t)=\eta(t)D_{0}\left ( \int_{0}^{t}S_{\R}^{t-t'}F(u,v)dt'\right ).
\end{align*}
Therefore, applying the smoothing estimates, we get 
\begin{align*}
    &\| u-S_{0}^{t}(u_{0},f) \|_{C_{t}H^{s+a}([0,T] \times \R^{+})}\lesssim \| u-\eta(t)S_{0}^{t}(u_{0},f) \|_{Y^{s+a,b} } \\
    &\qquad \lesssim 
    \left \|\eta(t)\int_{0}^{t}S_{\R}^{t-t'}F(u,v)dt' \right \|_{Y^{s+a,b} }+\|\eta(t)S_{0}^{t}(0,q)\|_{Y^{s+a,b}} \\
    &\qquad \lesssim \|u\|_{Y^{s,b}}\|v\|_{\t Y^{k,b}_{\alpha}}+\|u\|_{Y^{s,b} }^{3} <\infty.
\end{align*}

\subsection{Proof of Theorem \ref{Global}}
Let $f=0$ and $g=0$. If $(s,k)=(1,1)$, it is proved in \cite[Theorem 1.3]{CC1} that if $\alpha \gamma>0$, we have an a priori bound
\begin{align}\label{estimate: a priori bound}
    \|u\|_{H^{1}(\R^+)}+\|v\|_{H^{1}(\R^+)} \leq C(\|u_0\|_{H^{1}(\R^+)},\|v_0\|_{H^{1}(\R^+)})
\end{align}
and the solution is global. If $s\geq1$ and $k \geq 1$, we can extend the solution globally in time by the persistence of regularity. Polynomial growth of $\|u\|_{H^{s}(\R^+)}$ follows from (\ref{estimate: a priori bound}) and the nonlinear smoothing (\ref{smoothing}), see \cite{ET}.

\section{Uniqueness of high regularity solutions} \label{section: high regularity uniqueness}

In this section we discuss the uniqueness of solutions to (\ref{eq: SKdV}) at the regularity level $(\frac{3}{2}+,\frac{5}{2}+)$. 

The following argument is a variant of the one obtained in \cite{GMZ}. Assume that $(u_1, v_1)$ and $(u_2, v_2)$ are two solutions to (\ref{eq: SKdV}) with the same initial and boundary data. Then the difference $(u,v) := (u_1 -u_2, v_1 - v_2)$ satisfies
\begin{align}
    i\partial_{t}u + \partial_{x}^2 u = \alpha( uv_1 + u_2v) + \beta (u|u_1|^2 + (\overline{u}u_2 + \overline{u_1}u) u_2)  \label{difference_1}\\
    \partial_{t}v + \partial_{x}^3 v = \partial_{x}\left( \gamma(u\overline{u_1} + \overline{u}u_2) - \frac12 v(v_1 + v_2)\right)\label{difference_2}
\end{align}
on $(x,t) \in \R^+ \times \R^+$ with zero initial and boundary conditions.

We estimate the $H^{1}$-norms of $u$ and $v$. In this section, all implicit constants depend only on $\|u_{1}\|_{C_{t}H^{\frac{3}{2}+}_{x}}+\|u_{2}\|_{C_{t}H^{\frac{3}{2}+}_{x}}+\|v_{1}\|_{C_{t}H^{\frac{5}{2}+}_{x}}+\|v_{2}\|_{C_{t}H^{\frac{5}{2}+}_{x}}$. Also, in the following calculations, we assume that $(u_1, v_1)$ and $(u_2, v_2)$ are smooth. The argument extends to the general case using a limiting argument, see \cite{Holm} for details. 

First, we have
\begin{align*}
\partial_t  \| u\|_{L^2_{\R^+}}^2  
&= 
2 \Re \int_0^\infty \overline{u}u_t d x \\
&= 
2 \Im \int_0^\infty \overline{u}\left[ \alpha( uv_1 + u_2v) + \beta (u|u_1|^2 + (\overline{u}u_2 + \overline{u_1}u) u_2) - u_{xx}\right] d x.
\end{align*}
Integrating by parts and using the zero boundary condition for $u$, we obtain
\begin{align*}
\partial_t \| u\|_{L^2_{\R^+}}^2 
&= 
2 \Im \int_0^\infty \overline{u}\left[  \alpha( uv_1 + u_2v) + \beta (u|u_1|^2 + (\overline{u}u_2 + \overline{u_1}u) u_2\right] d x \\
&\leq
2 |\alpha| \left(\|u\|_{L^2_{\R^+}}^2 \|v_1\|_{H^{\frac12+}_{\R^+}} + \|u\|_{L^2_{\R^+}} \|v\|_{L^2_{\R^+}}  \|u_2\|_{H^{\frac12+}_{\R^+}}\right) \\
&\qquad + 2 |\beta|  \|u\|_{L^2_{\R^+}} ^2\left(\|u_1\|_{H^{\frac12+}_{\R^+}} + \|u_2\|_{H^{\frac12+}_{\R^+}}\right)^2 \\
&\lesssim \|u\|_{L^2_{\R^+}}^2+\|v\|_{L^2_{\R^+}}^2.
\end{align*}
We used the embedding $H^{\frac12+}(\R) \hookrightarrow L^{\infty}(\R)$. Also we have
\begin{align*}
\partial_t \| v\|_{L^2_{\R^+}}^2
&= 
 2\int_0^\infty   vv_t d x \\
&=  2 \int_0^\infty v [ \gamma(u\overline{u_1} + \overline{u}u_2) - \frac12 v(v_1 + v_2) - v_{xx} ]_{x} d x\\
&=  - v_x(0,t)^2 +  \int_0^\infty 2\gamma v\left(  u\overline{u_1} + \overline{u}u_2\right)_x - \frac12 v^2(v_1 + v_2)_x d x \\
&\leq
2 |\gamma| \left( \|u\|_{L^2_{\R^+}} \|v\|_{L^2_{\R^+}} \left(\|u_1\|_{H^{\frac32+}_{\R^+}} + \|u_2\|_{H^{\frac32+}_{\R^+}}\right) \right) \\
&\qquad + 2|\gamma| \left(\|u_1\|_{H^{\frac12+}_{\R^+}}\|u_{x}\|_{L^{2}_{\R^+}}\|v\|_{L^2_{\R^+}}+\|u_2\|_{H^{\frac12+}_{\R^+}}\|u_{x}\|_{L^{2}_{\R^+}}\|v\|_{L^2_{\R^+}} \right) \\
&\qquad + \frac12\|v\|_{L^2_{\R^+}} ^2\left(\|v_1\|_{H^{\frac32+}_{\R^+}} + \|v_2\|_{H^{\frac32+}_{\R^+}}\right) \\
&\lesssim \|u\|^{2}_{H^{1}_{\R^+}}+\|v\|^{2}_{H^{1}_{\R^+}}.
\end{align*}
Similarly, 
\begin{align*}
        \partial_t \| u_{x}\|_{L^2_{\R^+}}^2&=2\Re\int_{0}^{\infty}u_{xt}\overline{u}_{x}dx=-2\Re\int_{0}^{\infty}u_{t}\overline{u}_{xx}dx\\
        &=-2\Re\left( -i\int_{0}^{\infty}(\alpha( uv_1 + u_2v) + \beta (u|u_1|^2 + (\overline{u}u_2 + \overline{u_1}u) u_2)-u_{xx})\overline{u}_{xx}dx \right)\\
        &=-2\Im\left(\int_{0}^{\infty}(\alpha( uv_1 + u_2v) + \beta (u|u_1|^2 + (\overline{u}u_2 + \overline{u_1}u) u_2))\overline{u}_{xx}dx \right)\\
        &=2\Im\left(\int_{0}^{\infty}(\alpha( uv_1 + u_2v) + \beta (u|u_1|^2 + (\overline{u}u_2 + \overline{u_1}u) u_2))_{x}\overline{u}_{x}dx \right)\\ 
        &\lesssim \|u\|^{2}_{H^{1}_{\R^+}}+\|v\|^{2}_{H^{1}_{\R^+}}.
\end{align*}
Let $V(x,t):=-\int_{x}^{\infty}v(y,t)dy$. Integrating (\ref{difference_2}), we get 
\begin{align}
    V_{t}+v_{xx}=\gamma(u\overline{u_{1}}+\overline{u}u_{2})-\frac{1}{2}v(v_1+v_2). \label{uniqueness_V}
\end{align}
Note that 
\begin{align*}
    \int_{0}^{\infty}v_{t}V_{t}dx=-\frac{1}{2}V_{t}^{2}(0,t).
\end{align*}
Therefore, multiplying (\ref{uniqueness_V}) by $v_{t}$ and integrating, we have
\begin{align*}
\int_{0}^{\infty}v_{xx}v_{t}dx=\frac{1}{2}V_{t}^{2}(0,t)+\gamma\int_{0}^{\infty}(u\overline{u_{1}}+\overline{u}u_{2})v_{t}dx-\frac{1}{2}\int_{0}^{\infty}v v_{t}(v_1+v_2)dx.
\end{align*}
Let 
\begin{align*}
   L(t):=\frac{1}{2}\|v_{x}\|_{L^2_{\R^+}}^2+\gamma\int_{0}^{\infty}(u\overline{u_{1}}+\overline{u}u_{2})v dx-\frac{1}{4}\int_{0}^{\infty}v^{2}(v_1+v_2)dx.
\end{align*}
Since 
\begin{align*}
   \frac{1}{2}\partial_{t}\|v_{x}\|_{L^2_{\R^+}}^2 &=\int_{0}^{\infty}v_{x}v_{xt}dx 
   =-v_{x}(0,t)v_{t}(0,t)-\int_{0}^{\infty}v_{xx}v_{t}dx \\
   &=-\frac{1}{2}V_{t}^{2}(0,t)-\gamma\int_{0}^{\infty}(u\overline{u_{1}}+\overline{u}u_{2})v_{t}dx+\frac{1}{2}\int_{0}^{\infty}v v_{t}(v_1+v_2)dx,
\end{align*}
we have
\begin{align*}
   \frac{dL}{dt} 
   &=-\frac{1}{2}V_{t}^{2}(0,t)+\gamma\int_{0}^{\infty}(u\overline{u_{1}}+\overline{u}u_{2})_{t}vdx-\frac{1}{4}\int_{0}^{\infty}v^{2}(v_1+v_2)_{t}dx \\
   &=-\frac{1}{2}V_{t}^{2}(0,t)+\gamma\int_{0}^{\infty}(u_{t}\overline{u_{1}}+\overline{u}_{t}u_{2}+u\overline{u_{1}}_{t}+\overline{u}u_{2t})vdx\\
   & \qquad -\frac{1}{4}\int_{0}^{\infty}v^{2}(v_1+v_2)_{t}dx  \\
   &=:-\frac{1}{2}V_{t}^{2}(0,t)+\textit{I\,}+\textit{II\,}.
\end{align*}
We can bound $|\textit{I\,}|$ by $\|u\|^{2}_{H^{1}_{\R^+}}+\|v\|^{2}_{H^{1}_{\R^+}}$ using (\ref{eq: SKdV}), (\ref{difference_1}) and integration by parts. For example, the first term inside the integral $\textit{I\,}$ is
\begin{align*}
   \left|\int_{0}^{\infty}u_{t}\overline{u_{1}}vdx\right| &=\left|\int_{0}^{\infty}\overline{u_{1}}v(iu_{xx}-i\alpha( uv_1 + u_2v) -i \beta (u|u_1|^2+(\overline{u}u_2 + \overline{u_1}u) u_2))dx \right|\\
   &\leq \left|\int_{0}^{\infty}(\overline{u_{1}}_{x}v+\overline{u_{1}}v_{x})u_{x}dx \right| \\
   & \qquad +\left|\int_{0}^{\infty}{u_{1}}v(\alpha( uv_1 + u_2v)+ \beta (u|u_1|^2+(\overline{u}u_2 + \overline{u_1}u) u_2))dx \right| \\
   &\lesssim \|u\|^{2}_{H^{1}_{\R^+}}+\|v\|^{2}_{H^{1}_{\R^+}}.
\end{align*}
Also, since
\begin{align*}
     \textit{II\,}&=-\frac{1}{4}\left(\int_{0}^{\infty}v^{2}(\gamma(|u_{1}|^{2})_{x}-v_{1xxx}-v_{1}v_{1x})dx
      +\int_{0}^{\infty}v^{2}(\gamma(|u_{2}|^{2})_{x}-v_{2xxx}-v_{2}v_{2x})dx\right)\\
      &=-\frac{1}{4}\left(\int_{0}^{\infty}v^{2}(\gamma(|u_{1}|^{2})_{x}-v_{1}v_{1x}+\gamma(|u_{2}|^{2})_{x}-v_{2}v_{2x})dx
      +\int_{0}^{\infty}2vv_{x}(v_{1}+v_{2})_{xx}dx\right),
\end{align*}
we have $|\textit{II\,}|\lesssim \|u\|^{2}_{H^{1}_{\R^+}}+\|v\|^{2}_{H^{1}_{\R^+}}$. Therefore,
\begin{align*}
    L(t) \lesssim \int_{0}^{t}\|u\|^{2}_{H^{1}_{\R^+}}+\|v\|^{2}_{H^{1}_{\R^+}}dt'.
\end{align*}
Noting that 
\begin{align*}
      \left|\gamma\int_{0}^{\infty}(u\overline{u_{1}}+\overline{u}u_{2})v dx-\frac{1}{4}\int_{0}^{\infty}v^{2}(v_1+v_2)dx\right|
      \lesssim \|u\|^{2}_{L^{2}_{\R^+}}+\|v\|^{2}_{L^{2}_{\R^+}} 
      \lesssim \int_{0}^{t}\|u\|^{2}_{H^{1}_{\R^+}}+\|v\|^{2}_{H^{1}_{\R^+}}dt',
\end{align*}
we have 
\begin{align*}
    \|v_{x}\|_{L^2_{\R^+}}^2 \lesssim \int_{0}^{t}\|u\|^{2}_{H^{1}_{\R^+}}+\|v\|^{2}_{H^{1}_{\R^+}}dt'.
\end{align*}
Combining the estimates above, we obtain
\begin{align*}
    \|u\|^{2}_{H^{1}_{\R^+}}+\|v\|^{2}_{H^{1}_{\R^+}} \lesssim \int_{0}^{t}\|u\|^{2}_{H^{1}_{\R^+}}+\|v\|^{2}_{H^{1}_{\R^+}}dt',
\end{align*}
and the desired uniqueness follows from Gronwall's inequality.

\section{Proofs of nonlinear estimates} \label{section: proofs of estimates}
We provide the proofs of the nonlinear estimates presented in Section \ref{section: nonlinear estimates}. The following lemma is one of the main tool in proving our estimates.

\begin{lem} \label{lem: interpolation}
    For $f, g, h \in L^{2}_{\xi,\tau}$, we have
\begin{align*}
    \left | \iint_{\begin{subarray}{l}\\ \xi=\xi_{1}+\xi_{2}\\ \tau=\tau_{1}+\tau_{2} \end 
{subarray}} \lb 3\xi_2^2\pm2\xi_1 \rb^{\frac{1}{2}-}h(\xi, \tau) 
\frac{f(\xi_1, \tau_1)}{\lb \tau_1 \pm \xi_1^2 \rb^{\frac12 -}}\frac{g(\xi_{2}, \tau_{2})}{\lb \tau_2 - \xi_2^3 \rb^{\frac12 -}}  \right | \lesssim \| f\|_{L^2} \|g\|_{L^2} \|h\|_{L^2}
\end{align*}
and 
\begin{align*}
    \left | \iint_{\begin{subarray}{l}\\ \xi=\xi_{1}+\xi_{2}\\ \tau=\tau_{1}+\tau_{2} \end 
{subarray}} \lb \xi\rb^{\frac{1}{2}-}h(\xi, \tau) 
\frac{f(\xi_1, \tau_1)}{\lb \tau_1 + \xi_1^2 \rb^{\frac12 -}}\frac{g(\xi_{2}, \tau_{2})}{\lb \tau_2 \pm \xi_2^2 \rb^{\frac12 -}}  \right | \lesssim \| f\|_{L^2} \|g\|_{L^2} \|h\|_{L^2}.
\end{align*}
\end{lem}
Here $\frac12-$ means $\frac12-\epsilon$ for arbitrary $\epsilon \ll 1$. The implicit constants may depend on $\epsilon$.
\begin{proof}
Interpolating the $L^{6}_{x,t}$-Strichartz estimates 
\begin{align*} 
        \left \Vert \left [ \frac{f(\xi,\tau)}{\lb \tau \pm \xi^{2}\rb^{\frac{1}{2}+}}\right ]^{\vee}\right \Vert_{L^{6}_{x,t}} \lesssim \Vert f \Vert_{L^{2}_{\xi, \tau}} \quad \text {and} \quad \left \Vert \left [ \frac{\lb \xi \rb^{\frac{2}{3}}f(\xi,\tau)}{\lb \tau - \xi^{3}\rb^{\frac{1}{2}+}}\right ]^{\vee}\right \Vert_{L^{6}_{x,t}} \lesssim \Vert f \Vert_{L^{2}_{\xi, \tau}},
\end{align*}
\cite{KPV, T} with the Plancherel identity, we obtain
\begin{align} \label{eq: Strichartz interpolation}
        \left \Vert \left [ \frac{f(\xi,\tau)}{\lb \tau \pm \xi^{2}\rb^{\frac{3}{8}+}}\right ]^{\vee}\right \Vert_{L^{4}_{x,t}} \lesssim \Vert f \Vert_{L^{2}_{\xi, \tau}} \quad \text {and} \quad \left \Vert \left [ \frac{\lb \xi \rb^{\frac{1}{8}}f(\xi,\tau)}{\lb \tau - \xi^{3}\rb^{\frac{3}{8}+}}\right ]^{\vee}\right \Vert_{L^{4}_{x,t}} \lesssim \Vert f \Vert_{L^{2}_{\xi, \tau}}.
\end{align}
Therefore, we have
\begin{multline} \label{eq: multiplier 1 A}
    \left | \iint_{\begin{subarray}{l}\\ \xi=\xi_{1}+\xi_{2}\\ \tau=\tau_{1}+\tau_{2} \end 
{subarray}} h(\xi, \tau) 
\frac{f(\xi_1, \tau_1)}{\lb \tau_1 \pm \xi_1^2 \rb^{\frac38+}}\frac{g(\xi_{2}, \tau_{2})}{\lb \tau_2 - \xi_2^3 \rb^{\frac38 +}}  \right |  \\
\leq \| h \|_{L^2} \left \Vert \left [ \frac{f(\xi,\tau)}{\lb \tau + \xi^{2}\rb^{\frac38+}}\right ]^{\vee}\right \Vert_{L^{4}_{x,t}} \left \Vert \left [ \frac{g(\xi,\tau)}{\lb \tau - \xi^{3}\rb^{\frac38+}}\right ]^{\vee}\right \Vert_{L^{4}_{x,t}} \lesssim \| f\|_{L^2} \|g\|_{L^2} \|h\|_{L^2}.
\end{multline}
and similarly
\begin{align} \label{eq: multiplier 1 B}
    \left | \iint_{\begin{subarray}{l}\\ \xi=\xi_{1}+\xi_{2}\\ \tau=\tau_{1}+\tau_{2} \end 
{subarray}} h(\xi, \tau) 
\frac{f(\xi_1, \tau_1)}{\lb \tau_1 + \xi_1^2 \rb^{\frac38+}}\frac{g(\xi_{2}, \tau_{2})}{\lb \tau_2 \pm \xi_2^2 \rb^{\frac38 +}}  \right | \lesssim \| f\|_{L^2} \|g\|_{L^2} \|h\|_{L^2}.
\end{align}
It is proved in \cite[Lemma 2.2]{Wu} that
\begin{align}\label{eq: multiplier m A}
    \left | \iint_{\begin{subarray}{l}\\ \xi=\xi_{1}+\xi_{2}\\ \tau=\tau_{1}+\tau_{2} \end 
{subarray}} |3\xi_2^2\pm2\xi_1 |^{\frac{1}{2}}h(\xi, \tau) 
\frac{f(\xi_1, \tau_1)}{\lb \tau_1 \pm \xi_1^2 \rb^{\frac12 +}}\frac{g(\xi_{2}, \tau_{2})}{\lb \tau_2 - \xi_2^3 \rb^{\frac12 +}}  \right | \lesssim \| f\|_{L^2} \|g\|_{L^2} \|h\|_{L^2}
\end{align}
and 
\begin{align} \label{eq: multiplier m B}
    \left | \iint_{\begin{subarray}{l}\\ \xi=\xi_{1}+\xi_{2}\\ \tau=\tau_{1}+\tau_{2} \end 
{subarray}} | \xi|^{\frac{1}{2}}h(\xi, \tau) 
\frac{f(\xi_1, \tau_1)}{\lb \tau_1 + \xi_1^2 \rb^{\frac12+}}\frac{g(\xi_{2}, \tau_{2})}{\lb \tau_2 \pm \xi_2^2 \rb^{\frac12 +}}  \right | \lesssim \| f\|_{L^2} \|g\|_{L^2} \|h\|_{L^2}.
\end{align}
Interpolating trilinear functionals (\ref{eq: multiplier 1 A}) and (\ref{eq: multiplier m A}), we obtain
\begin{align} \label{eq: multiplier m C}
    \left | \iint_{\begin{subarray}{l}\\ \xi=\xi_{1}+\xi_{2}\\ \tau=\tau_{1}+\tau_{2} \end 
{subarray}} | 3\xi_2^2\pm2\xi_1 |^{\frac{1}{2}-}h(\xi, \tau) 
\frac{f(\xi_1, \tau_1)}{\lb \tau_1 \pm \xi_1^2 \rb^{\frac12 -}}\frac{g(\xi_{2}, \tau_{2})}{\lb \tau_2 - \xi_2^3 \rb^{\frac12 -}}  \right | \lesssim \| f\|_{L^2} \|g\|_{L^2} \|h\|_{L^2},
\end{align}
see \cite{GM}. Using (\ref{eq: multiplier 1 A}) and the triangle inequality, we can replace the multiplier $| 3\xi_2^2\pm2\xi_1 |^{\frac{1}{2}-}$ in (\ref{eq: multiplier m C}) with $\lb 3\xi_2^2\pm2\xi_1 \rb^{\frac{1}{2}-}$.
Using (\ref{eq: multiplier 1 B}) and (\ref{eq: multiplier m B}), we similarly obtain
\begin{align*}
    \left | \iint_{\begin{subarray}{l}\\ \xi=\xi_{1}+\xi_{2}\\ \tau=\tau_{1}+\tau_{2} \end 
{subarray}} \lb \xi\rb^{\frac{1}{2}-}h(\xi, \tau) 
\frac{f(\xi_1, \tau_1)}{\lb \tau_1 + \xi_1^2 \rb^{\frac12 -}}\frac{g(\xi_{2}, \tau_{2})}{\lb \tau_2 - \xi_2^2 \rb^{\frac12 -}}  \right | \lesssim \| f\|_{L^2} \|g\|_{L^2} \|h\|_{L^2}.
\end{align*}
\end{proof}

In the following proofs we always assume that $b<\frac{1}{2}$ is sufficiently close to $\frac{1}{2}$.

\subsection{Proof of Proposition \ref{estimate:uv}} \label{proof: uv}

By duality, it suffices to bound
\begin{align} \label{eq: uv dual}
\iint_{\begin{subarray}{l}\\ \xi=\xi_{1}+\xi_{2}\\ \tau=\tau_{1}+\tau_{2} \end
{subarray}} \frac{\lb \xi \rb^{s+a}|h(\xi, \tau)|}{\lb \tau + \xi^2 \rb^b}\frac{|f(\xi_1, \tau_1)|}{\lb \xi_{1} \rb^{s}\lb \tau_1 + \xi_1^2 \rb^b}\frac{|g(\xi_{2}, \tau_{2})|}{\lb \xi_{2} \rb^{k}\lb \tau_2 - \xi_2^3 \rb^b}    
\end{align}
by $\| f\|_{L^2} \|g\|_{L^2} \|h\|_{L^2}$.
Applying Cauchy-Schwarz inequality, we can bound (\ref{eq: uv dual}) by
\begin{align*}
        & \int_{\R^4} \frac{\lb \xi \rb^{s+a}|h(\xi, \tau)|}{\lb \tau + \xi^2 \rb^b}\frac{|f(\xi-\xi_2, \tau-\tau_2)|}{\lb \xi-\xi_2 \rb^{s}\lb (\tau-\tau_2) + (\xi-\xi_2)^2 \rb^b}\frac{|g(\xi_{2}, \tau_{2})|}{\lb \xi_{2} \rb^{k}\lb \tau_2 - \xi_2^3 \rb^b} d\xi d\xi_2 d\tau d\tau_2 \\
        & \leq  \left (\int_{\R^4} |f(\xi-\xi_2, \tau-\tau_2)|^2 |g(\xi_{2}, \tau_{2})|^2 d\xi d\xi_2 d\tau d\tau_2\right)^{\frac12} \\
        & \qquad \times \left (\int_{\R^4} \frac{\lb \xi \rb^{2s+2a}|h(\xi, \tau)|^2}{\lb \tau + \xi^2 \rb^{2b}\lb \xi-\xi_2 \rb^{2s}\lb (\tau-\tau_2) + (\xi-\xi_2)^2 \rb^{2b}\lb \xi_{2} \rb^{2k}\lb \tau_2 - \xi_2^3 \rb^{2b}} d\xi d\xi_2 d\tau d\tau_2\right )^{\frac12}  \\
        & \lesssim \| f\|_{L^2} \|g\|_{L^2} \left (\int_{\R^3} \frac{\lb \xi \rb^{2s+2a}|h(\xi, \tau)|^2}{\lb \tau + \xi^2 \rb^{2b}\lb \xi-\xi_2 \rb^{2s} \lb \xi_{2} \rb^{2k}\lb \tau+ (\xi-\xi_2)^2- \xi_2^3 \rb^{4b-1}} d\xi d\xi_2 d\tau \right )^{\frac12} \\
        & \lesssim \| f\|_{L^2} \|g\|_{L^2} \left (\int_{\R^3} \frac{\lb \xi \rb^{2s+2a}|h(\xi, \tau)|^2}{\lb \xi-\xi_2 \rb^{2s} \lb \xi_{2} \rb^{2k}\lb \xi_2 ( \xi_{2}^2-\xi_2+2\xi)\rb^{4b-1}} d\xi d\xi_2 d\tau \right )^{\frac12} \\
        & \leq \left ( \sup_{\xi} \int \frac{\lb \xi \rb^{2s+2a}}{\lb \xi-\xi_2 \rb^{2s} \lb \xi_{2} \rb^{2k}\lb \xi_2 ( \xi_{2}^2-\xi_2+2\xi)\rb^{4b-1}} d\xi_2 \right )^{\frac12} \| f\|_{L^2} \|g\|_{L^2} \|h\|_{L^2}.
\end{align*}
We used Lemma \ref{CalcLem} in $\tau_2$ integration in the second inequality, and 
\begin{multline*}
    \lb \tau + \xi^2 \rb^{2b}\lb \tau+ (\xi-\xi_2)^2- \xi_2^3 \rb^{4b-1} \geq \lb \tau + \xi^2 \rb^{4b-1}\lb \tau+ (\xi-\xi_2)^2- \xi_2^3 \rb^{4b-1} \\ \gtrsim \lb (\tau + \xi^2 )-(\tau+ (\xi-\xi_2)^2- \xi_2^3) \rb^{4b-1} = \lb \xi_2 ( \xi_{2}^2-\xi_2+2\xi)\rb^{4b-1}
\end{multline*}
in the third inequality. Therefore, it suffices to show that the quantity
\begin{align} \label{eq: uv reduced}
    \sup_{\xi} \int \frac{\lb \xi \rb^{2s+2a}}{\lb \xi-\xi_2 \rb^{2s} \lb \xi_{2} \rb^{2k}\lb \xi_2 ( \xi_{2}^2-\xi_2+2\xi)\rb^{4b-1}} d\xi_2 
\end{align}
is finite. 

\noindent \textit{Case 1: $|\xi_{2}| \lesssim 1$.}

Note that, in this case
\begin{align*}
    \lb\xi_{2}(\xi_{2}^{2}-\xi_{2}+2\xi)\rb \gtrsim \lb \xi_{2} \xi \rb \gtrsim |\xi_{2}|\lb \xi \rb.
\end{align*}
Using this, we estimate (\ref{eq: uv reduced}) by
\begin{align*}
        \sup_{\xi} \int_{|\xi_{1}| \lesssim 1} \frac{\lb \xi \rb^{2s+2a}d\xi_{1}}{\lb \xi \rb^{2s+4b-1}|\xi_{1}|^{4b-1}} \lesssim \sup_{\xi} \frac{1}{\lb \xi \rb^{4b-1-2a}}.
\end{align*}

\noindent \textit{Case 2: $|\xi| \lesssim 1$.}

Since $\lb\xi_{2}(\xi_{2}^{2}-\xi_{2}+2\xi)\rb \gtrsim \lb \xi_{2} \rb^{3}$ and $\lb\xi_{2}-\xi \rb \approx \lb\xi_{2}\rb$, we have
\begin{align*}
        (\ref{eq: uv reduced}) &\lesssim \sup_{\xi} \int \frac{d\xi_{2}}{\lb \xi_{2} \rb^{2s+2k+3(4b-1)}}.
\end{align*}

In the rest of this proof we shall assume that $|\xi| \gg 1$ and $|\xi_2| \gg 1$.

\noindent \textit{Case 3: $\xi \leq -1$.}

Let $M=\max(|\tau + \xi^2|,|\tau_1 + \xi_1^2|,|\tau_2 - \xi_2^3|)$. Then 
\begin{align*}
    M \gtrsim |(\tau + \xi^2)-(\tau_1 + \xi_1^2)-(\tau_2 - \xi_2^3)|=|\xi_2(\xi_{2}^{2}-\xi_{2}+2\xi)|.
\end{align*}
Since we are assuming $\xi \leq 1$, we have $|\xi_{2}^{2}-\xi_{2}+2\xi|=|(\xi_2-\frac12)^{2}+2\xi-\frac14|\gtrsim \max(|\xi_2|^2,|\xi|)$. Therefore, $M \gtrsim |\xi_2|^{3}$ and also $M \gtrsim |\xi_2||\xi_{2}^{2}-\xi_{2}+2\xi|^{1-\frac{a}{b}}|\xi_{2}^{2}-\xi_{2}+2\xi|^{\frac{a}{b}} \gtrsim |\xi_2|^{3-2\frac{a}{b}}|\xi|^{\frac{a}{b}}$.
Since $|\xi| \gg 1$ and $|\xi_2| \gg 1$, we have $\lb M \rb \gtrsim \lb \xi_2 \rb^{3-2\frac{a}{b}}\lb \xi \rb^{\frac{a}{b}}$. 
Let 
\begin{align*}
    P(\xi,\tau,\xi_1,\tau_1,\xi_2,\tau_2):=\frac{\lb \xi \rb^{s+a}|h(\xi, \tau)|}{\lb \tau + \xi^2 \rb^b}\frac{|f(\xi_1, \tau_1)|}{\lb \xi_{1} \rb^{s}\lb \tau_1 + \xi_1^2 \rb^b}\frac{|g(\xi_{2}, \tau_{2})|}{\lb \xi_{2} \rb^{k}\lb \tau_2 - \xi_2^3 \rb^b}.
\end{align*}
We separately consider the case $s \geq 0$ and the case $s < 0$. Our main tool is Lemma \ref{lem: interpolation}.

\textit{Case 3A: $s \geq 0$.}

(i) $M=|\tau + \xi^2|$. In this case,
\begin{align*}
      P(\xi,\tau,\xi_1,\tau_1,\xi_2,\tau_2) 
      \lesssim \frac{\lb \xi \rb^{s}}{\lb \xi_1 \rb^{s}\lb \xi_2 \rb^{k+3b-2a}}|h(\xi, \tau)|\frac{|f(\xi_1, \tau_1)|}{\lb \tau_1 + \xi_1^2 \rb^b}\frac{|g(\xi_{2}, \tau_{2})|}{\lb \tau_2 - \xi_2^3 \rb^b}. 
\end{align*}
If $|\xi_1| \lesssim |\xi_2|$, then $\lb 3\xi_2^2+2\xi_1 \rb \approx \lb \xi_2\rb^2$. Using $\lb \xi \rb^{s} \lesssim \lb \xi_1 \rb^{s}\lb \xi_2 \rb^{s}$,
\begin{align*}
    P(\xi,\tau,\xi_1,\tau_1,\xi_2,\tau_2) 
    & \lesssim \lb \xi_2 \rb^{s-k-3b+2a}|h(\xi, \tau)|\frac{|f(\xi_1, \tau_1)|}{\lb \tau_1 + \xi_1^2 \rb^b}\frac{|g(\xi_{2}, \tau_{2})|}{\lb \tau_2 - \xi_2^3 \rb^b} \\
    & \lesssim \lb 3\xi_2^2+2\xi_1 \rb^{\frac{1}{2}(s-k-3b+2a)}|h(\xi, \tau)|\frac{|f(\xi_1, \tau_1)|}{\lb \tau_1 + \xi_1^2 \rb^b}\frac{|g(\xi_{2}, \tau_{2})|}{\lb \tau_2 - \xi_2^3 \rb^b}.
\end{align*}
If $|\xi_1| \gg |\xi_2|$, then $|\xi_1| \approx |\xi|$ and
\begin{align*}
     P(\xi,\tau,\xi_1,\tau_1,\xi_2,\tau_2) \lesssim |h(\xi, \tau)|\frac{|f(\xi_1, \tau_1)|}{\lb \tau_1 + \xi_1^2 \rb^b}\frac{|g(\xi_{2}, \tau_{2})|}{\lb \tau_2 - \xi_2^3 \rb^b}.
\end{align*}
Comparing these with Lemma \ref{lem: interpolation}, the desired estimate follows.

(ii) $M=|\tau_1 + \xi_1^2|$. This case is similar to (i). We have
\begin{align*}
      P(\xi,\tau,\xi_1,\tau_1,\xi_2,\tau_2) 
      \lesssim \frac{\lb \xi \rb^{s}}{\lb \xi_1 \rb^{s}\lb \xi_2 \rb^{k+3b-2a}}|f(\xi_1, \tau_1)|\frac{|h(\xi, \tau)|}{\lb \tau + \xi^2 \rb^b}\frac{|g(\xi_{2}, \tau_{2})|}{\lb \tau_2 - \xi_2^3 \rb^b}. 
\end{align*}
If $|\xi| \lesssim |\xi_2|$, then use $\lb 3\xi_2^2+2\xi \rb \approx \lb \xi_2\rb^2$ and $\lb \xi \rb^{s} \lesssim \lb \xi_1 \rb^{s}\lb \xi_2 \rb^{s}$. If $|\xi| \gg |\xi_2|$, then $|\xi_1| \approx |\xi|$, thus the multiplier is bounded.

(iii) $M=|\tau_2 + \xi_2^3|$. We have
\begin{align*}
      P(\xi,\tau,\xi_1,\tau_1,\xi_2,\tau_2) 
      \lesssim \frac{\lb \xi \rb^{s}}{\lb \xi_1 \rb^{s}\lb \xi_2 \rb^{k+3b-2a}}|g(\xi_{2}, \tau_{2})|\frac{|h(\xi, \tau)|}{\lb \tau + \xi^2 \rb^b}\frac{|f(\xi_1, \tau_1)|}{\lb \tau_1 + \xi_1^2 \rb^b}. 
\end{align*}
Since $\lb \xi \rb^s \lesssim \lb \xi_1 \rb^s \lb \xi_2 \rb^s$, 
\begin{align*}
      P(\xi,\tau,\xi_1,\tau_1,\xi_2,\tau_2) 
      \lesssim \lb \xi_2 \rb^{s-k-3b+2a}|g(\xi_{2}, \tau_{2})|\frac{|h(\xi, \tau)|}{\lb \tau + \xi^2 \rb^b}\frac{|f(\xi_1, \tau_1)|}{\lb \tau_1 + \xi_1^2 \rb^b}. 
\end{align*}

\textit{Case 3B: $s < 0$.}

Using $\lb \xi_1 \rb^{-s} \lesssim \lb \xi \rb^{-s} \lb \xi_2 \rb^{-s}$, we have
\begin{align*}
    \frac{\lb \xi \rb^{s}}{\lb \xi_1 \rb^{s}\lb \xi_2 \rb^{k+3b-2a}} \lesssim \lb \xi_2 \rb^{-s-k-3b+2a} \lesssim 1
\end{align*}
provided $-s-k-3b+2a \leq 0$.

\noindent \textit{Case 4: $\xi<-1$ and $|\xi_{2}| \geq 1+\sqrt{1-8\xi}$.}

In this case, we have 
\begin{align*}
    \left|\xi_{2}-\frac{1 \pm \sqrt{1-8\xi}}{2}\right| \geq |\xi_{2}|-\frac{1+ \sqrt{1-8\xi}}{2} \geq \frac{|\xi_{2}|}{2}.
\end{align*}
Since we are assuming $\xi<-1$ and $|\xi_{2}| \geq 1+\sqrt{1-8\xi}$ and $|\xi| \gg 1$ and $|\xi_2| \gg 1$, 
\begin{align*}
     \left|\xi_{2}-\frac{1 \pm \sqrt{1-8\xi}}{2}\right| 
    \gtrsim \max(| \xi_{2} |,| \xi |^{\frac{1}{2}}).
\end{align*}
Therefore, 
\begin{align*}
| \xi_{2}^{2}-\xi_{2}+2\xi | = \left |  \left(\xi_{2}-\frac{1 + \sqrt{1-8\xi}}{2}\right) \left(\xi_{2}-\frac{1 - \sqrt{1-8\xi}}{2}\right) \right | \gtrsim \max(|\xi_2|^2,|\xi|).
\end{align*}
Using this, we can estimate as in Case 3.

\noindent \textit{Case 5: $\xi<-1$ and $|\xi_{2}| \leq \frac{\sqrt{1-8\xi}-1}{4}$.}

In this case, we have 
\begin{align*}
    \left|\xi_{2}-\frac{1 \pm \sqrt{1-8\xi}}{2}\right| 
    \geq \frac{1+\sqrt{1-8\xi}}{2}-|\xi_{2}|
    \gtrsim \max(| \xi_{2} |,| \xi |^{\frac{1}{2}}).
\end{align*}
which is the same as in Case 4.

\noindent \textit{Case 6: $\xi<-1$ and $|\xi_{2}-\frac{1+\sqrt{1-8\xi}}{2}| \leq 1$.}

We have 
\begin{align*}
        (\ref{eq: uv reduced}) &\lesssim \sup_{\xi} \int\frac{\lb \xi \rb^{2s+2a} d\xi_{2}}{\lb \xi -\xi_{2} \rb^{2s}\lb \xi_{2} \rb^{2k}\lb \xi_{2}(\xi_{2}-\frac{1 + \sqrt{1-8\xi}}{2}) \rb^{4b-1}| \xi_{2}-\frac{1 - \sqrt{1-8\xi}}{2} |^{4b-1}} \\
        &\lesssim \sup_{\xi} \frac{\lb \xi \rb^{2s+2a}}{\lb \xi \rb^{2s+k+4b-1}} \int_{|\xi_{2}-\frac{1+\sqrt{1-8\xi}}{2}| \leq 1} \frac{d\xi_{2}}{|\xi_{2}-\frac{1+\sqrt{1-8\xi}}{2}|^{4b-1}}\\
        &< \infty.
\end{align*}

\noindent \textit{Case 7: $\xi<-1$ and $|\xi_{2}-\frac{1-\sqrt{1-8\xi}}{2}| \leq 1$.}

Similar to Case 6. 

\noindent \textit{Case 8: $\xi<-1$, $|\xi_{2}-\frac{1\pm \sqrt{1-8\xi}}{2}| \geq 1$ and $\frac{\sqrt{1-8\xi}-1}{4} \leq |\xi_{2}| \leq 1+\sqrt{1-8\xi}$.}

Since $|\xi_{2}| \approx |\xi|^{\frac{1}{2}}$, we have 
\begin{align*}
        (\ref{eq: uv reduced}) &\lesssim \sup_{\xi} \int \frac{\lb \xi \rb^{2s+2a}d\xi_{2}}{\lb \xi - \xi_{2} \rb^{2s}\lb \xi_{2} \rb^{2k}\lb \xi_{2}-\frac{1+\sqrt{1-8\xi}}{2} \rb^{4b-1}\lb \xi_{2} \rb^{4b-1}\lb \xi_{2}-\frac{1-\sqrt{1-8\xi}}{2} \rb^{4b-1}} \\ & \lesssim \sup_{\xi} \frac{1}{\lb \xi \rb ^{k-2a+\frac{1}{2}(4b-1)}}\int \frac{d\xi_{2}}{\lb \xi_{2}-\frac{1 + \sqrt{1-8\xi}}{2} \rb^{4b-1}\lb \xi_{2}-\frac{1 - \sqrt{1-8\xi}}{2} \rb^{4b-1}} \\
        & \lesssim \sup_{\xi} \frac{1}{\lb \xi \rb ^{k-2a+\frac{3}{2}(4b-1)-\frac{1}{2}}}.  
\end{align*}

\subsection{Proof of Proposition \ref{estimate:vv_x}}

For the $s>0$ case, see \cite{ET}. If $s=0$, it suffices to bound the quantity
    \begin{align}
        \iint_{\begin{subarray}{l}\\ \xi=\xi_{1}+\xi_{2}+\xi_{3}\\ \tau=\tau_{1}+\tau_{2}+\tau_{3} \end
{subarray}}\frac{h(\xi,\tau)}{\lb \tau + \xi^{2}\rb^{b}}\frac{f(\xi_{1},\tau_{1})}{\lb \tau_1 + \xi_1^{2}\rb^{b}}\frac{g(\xi_{2},\tau_{2})}{\lb \tau_2 - \xi_2^{2}\rb^{b}}\frac{k(\xi_{3},\tau_{3})}{\lb \tau_3 + \xi_3^{2}\rb^{b}}  \label{eq:kdvdual}
    \end{align}
    by $\Vert f \Vert_{L^{2}_{\xi,\tau}}\Vert g \Vert_{L^{2}_{\xi,\tau}}\Vert h \Vert_{L^{2}_{\xi,\tau}}\Vert k \Vert_{L^{2}_{\xi,\tau}}$.
Assuming $b>\frac38$, this immediately follows from (\ref{eq: Strichartz interpolation}) and the H\"older inequality.

\subsection{Proof of Proposition \ref{estimate:|u|^2_x}} \label{proof: |u|^2_x}

By duality, it is enough to bound the quantity
\begin{align}
         \iint_{\begin{subarray}{l}\\ \xi=\xi_{1}+\xi_{2}\\ \tau=\tau_{1}+\tau_{2} \end{subarray}} \frac{\lb \xi \rb^{k+a}|\xi||h(\xi,\tau)|}{\lb \tau- \xi^{3} \rb^{b}}\frac{|f(\xi_1,\tau_1)|}{\lb \xi_{1} \rb^{s}\lb \tau_{1}+\xi_{1}^{2} \rb^{b}}\frac{|g(\xi_{2},\tau_{2})|}{\lb \xi_{2}\rb^{s}\lb \tau_{2}- \xi_{2}^{2}\rb^{b}}
        \label{kdv_duality}
\end{align}
by $\left \Vert f \right \Vert_{L^{2}}\left \Vert g \right \Vert_{L^{2}}\left \Vert h \right \Vert_{L^{2}}$. Arguing as in the proof of Proposition \ref{estimate:uv}, it suffices to demonstrate that the quantity
\begin{align}
        \sup_{\xi_{1}}\frac{1}{\lb \xi_{1} \rb^{2s}}\int\frac{\lb \xi \rb^{2k+2a+2} d\xi}{\lb \xi -\xi_{1} \rb^{2s}\lb\xi(\xi^{2}+\xi-2\xi_{1})\rb^{4b-1}} \label{kdv_reduced}
\end{align}
is finite. 

\noindent \textit{Case 1: $|\xi| \lesssim 1$.}

In this case, we have $\lb\xi(\xi^{2}+\xi-2\xi_{1})\rb \gtrsim \lb \xi \xi_1 \rb \gtrsim |\xi|\lb \xi_1 \rb$ . Thus,
\begin{align*}
    (\ref{kdv_reduced}) \lesssim \sup_{\xi_{1}} \frac{1}{\lb \xi_{1} \rb^{4s+4b-1}}\int_{|\xi| \lesssim 1} \frac{1}{|\xi|^{4b-1}} d\xi.
\end{align*}

In the rest of this proof we always assume that $|\xi| \gg 1$. 

\noindent \textit{Case 2: $\xi_{1} < 1$.}

Let $M=\max(|\tau_{1}- \xi_{1}^{2}|,|\tau_{2}+ \xi_{2}^{2}|,|\tau- \xi^{3}|)$. Then $M \gtrsim |\xi(\xi^{2}+\xi-2\xi_{1})|$. Since $\xi_{1} < 1$ and $|\xi| \gg 1$, we have $\xi^{2}+\xi-2\xi_{1}=(\xi+\frac{1}{2})^{2}-2\xi_{1}-\frac{1}{4} \gtrsim |\xi|^{2}$, therefore $M \gtrsim |\xi|^{3}$. Let
\begin{align*}
    P(\xi,\tau,\xi_1,\tau_1,\xi_2,\tau_2):=\frac{\lb \xi \rb^{k+a}|\xi||h(\xi,\tau)|}{\lb \tau- \xi^{3} \rb^{b}}\frac{|f(\xi_1,\tau_1)|}{\lb \xi_{1} \rb^{s}\lb \tau_{1}+\xi_{1}^{2} \rb^{b}}\frac{|g(\xi_{2},\tau_{2})|}{\lb \xi_{2}\rb^{s}\lb \tau_{2}- \xi_{2}^{2}\rb^{b}}.
\end{align*}
In the following, we will frequently refer to Lemma \ref{lem: interpolation}.

\noindent \textit{Case 2A: $s \geq 0$}

(i) $M=|\tau- \xi^{3}|$. Using $\lb \xi_1 \rb^s \lb \xi_2 \rb^s \gtrsim \lb \xi \rb^s$ we have
\begin{align*}
    P(\xi,\tau,\xi_1,\tau_1,\xi_2,\tau_2)
     & \lesssim \frac{\lb \xi \rb^{k+a+1-3b}}{\lb \xi_1 \rb^s \lb \xi_2 \rb^s}|h(\xi,\tau)| \frac{|f(\xi_1,\tau_1)|}{\lb \tau_1+\xi_1^{2} \rb^{b}}\frac{|g(\xi_2,\tau_2)|}{\lb \tau_2- \xi_2^{2} \rb^{b}} \\
    & \lesssim \lb \xi \rb^{k-s+a+1-3b}|h(\xi,\tau)| \frac{|f(\xi_1,\tau_1)|}{\lb \tau_1+\xi_1^{2} \rb^{b}}\frac{|g(\xi_2,\tau_2)|}{\lb \tau_2- \xi_2^{2} \rb^{b}}.
\end{align*}

(ii) $M=|\tau_1 +\xi_1^2|$. In this case,
\begin{align*}
    P(\xi,\tau,\xi_1,\tau_1,\xi_2,\tau_2)
     & \lesssim \frac{\lb \xi \rb^{k+a+1-3b}}{\lb \xi_1 \rb^s \lb \xi_2 \rb^s}|f(\xi_1,\tau_1)| \frac{|g(\xi_2,\tau_2)|}{\lb \tau_2- \xi_2^{2} \rb^{b}}\frac{|h(\xi,\tau)|}{\lb \tau-\xi^3 \rb^{b}}.
\end{align*}
If $|\xi_2| \lesssim |\xi|^{2-}$, using $\lb \xi_1 \rb^s \lb \xi_2 \rb^s \gtrsim \lb \xi \rb^s$ and $\lb 3\xi^2-2\xi_2 \rb \approx \lb \xi \rb^2$ we obtain
\begin{align*}
    P(\xi,\tau,\xi_1,\tau_1,\xi_2,\tau_2)
     & \lesssim \lb 3\xi^2-2\xi_1 \rb^{\frac12(k-s+a+1-3b)}|f(\xi_1,\tau_1)| \frac{|g(\xi_2,\tau_2)|}{\lb \tau_2- \xi_2^{2} \rb^{b}}\frac{|h(\xi,\tau)|}{\lb \tau-\xi^3 \rb^{b}}.
\end{align*}
If $|\xi_2| \gg |\xi|^{2-}$, then $|\xi_2| \approx |\xi_1|$ and
\begin{align*}
    \frac{\lb \xi \rb^{k+a+1-3b}}{\lb \xi_1 \rb^s \lb \xi_2 \rb^s} \approx\lb \xi \rb^{k+a+1-3b}\lb \xi_2 \rb^{-2s} \lesssim \lb \xi \rb^{k+a+1-3b}\lb \xi \rb^{-4s+}.
\end{align*}
Therefore, the multiplier is bounded provided $k-4s+a+1-3b<0$.

(iii) $M=|\tau_2 -\xi_2^2|$. This case is almost identical to (ii).

\textit{Case 2B: $s < 0$}

(i) $|\xi| \gtrsim |\xi_1|$. In this case, using $|\xi_2| \leq |\xi|+|\xi_1| \lesssim |\xi|$ we have
\begin{align*}
    \frac{\lb \xi \rb^{k+a+1-3b}}{\lb \xi_1 \rb^{s}\lb \xi_2 \rb^{s}} \lesssim \lb \xi \rb^{k-2s+a+1-3b}  \lesssim 1
\end{align*}
provided $k-2s+a+1-3b \leq 0$.

(ii) $|\xi| \ll |\xi_1|$. We look at \eqref{kdv_reduced}. For this case, the supremum there reduces to 
\begin{align*}
\sup_{\xi_1}\;\; \lb \xi_1 \rb^{-4s} \int_{|\xi| \ll |\xi_1|} \frac{\lb \xi \rb^{2k +2a + 3 - 4b} \d \xi}{\lb \xi^2 + \xi - 2\xi_1 \rb^{4b-1}} 
\lesssim 
\sup_{\xi_1}\;\; \lb \xi_1 \rb^{-4s} \int_{|\rho| \ll |\xi_1|^2} \frac{\lb \rho + 2 \xi_1 \rb^{k +a + 1 - 2b} \d \rho}{\lb \rho \rb^{4b-1}},
\end{align*} 
where we have changed variables by setting $\rho = \xi^2 + \xi - 2 \xi_1$, so that $\d \rho = (2 \xi + 1) \d \xi = \pm \sqrt{1 +4(\rho + 2 \xi_1)} \d \xi$. 

In the region where $|\rho + 2 \xi_1| \lesssim |\xi_1|$, as long as  $k + a > 2b - 2$, we arrive at 
\begin{align*}
\sup_{\xi_1}\;\; \lb \xi_1 \rb^{-4s + k + a + 3 - 6b}.
\end{align*} 

It remains to bound
\begin{multline*}
\sup_{\xi_1}\;\; \lb \xi_1 \rb^{-4s} \int_{|\xi_1| \ll |\rho| \ll |\xi_1|^2} \frac{\lb \rho + 2 \xi_1 \rb^{k +a + 1 - 2b} \d \rho}{\lb \rho \rb^{4b-1}} \\
\approx
\sup_{\xi_1}\;\; \lb \xi_1 \rb^{-4s} \int_{|\xi_1| \ll |\rho| \ll |\xi_1|^2}\lb \rho \rb^{k +a + 2 - 6b} \d \rho.
\end{multline*} 
This is finite as long as $k+ a+2-6b < -1$ and $a<4s - k + 6b - 3$.

\noindent \textit{Case 3: $\xi_{1} \geq 1$ and $|\xi| \geq 1+\sqrt{1+8\xi_{1}}$ or $|\xi| \leq \frac{\sqrt{1+8\xi_{1}}-1}{4}$.}

In this case $|\xi(\xi^{2}+\xi-2\xi_{1})| \gtrsim |\xi|^{3}$ and we can argue similarly to Case 2. 

\noindent \textit{Case 4: $\xi_{1} \geq 1$, $|\xi-\frac{-1-\sqrt{1+8\xi_{1}}}{2}| \leq 1$ and $\frac{\sqrt{1+8\xi_{1}}-1}{4} \leq |\xi| \leq 1+\sqrt{1+8\xi_{1}}$.}

In this case, we have
\begin{align*}
     \left|\xi (\xi-\frac{-1+\sqrt{1+8\xi_{1}}}{2})\right| \gtrsim |\xi||\xi_{1}|^{\frac{1}{2}} \gtrsim |\xi_{1}|.
\end{align*}
Using this we have
\begin{align*}
    ( \ref{kdv_reduced} )
        & \lesssim \sup_{\xi_{1}}\frac{1}{\lb \xi_{1} \rb^{4s-k-a-1}}\int_{|\xi-\frac{-1-\sqrt{1+8\xi_{1}}}{2}| \leq 1}\frac{d\xi}{\lb \xi (\xi-\frac{-1+\sqrt{1+8\xi_{1}}}{2}) \rb^{4b-1}|\xi-\frac{-1-\sqrt{1+8\xi_{1}}}{2}|^{4b-1}} \\
        & \lesssim \sup_{\xi_{1}}\frac{1}{\lb \xi_{1} \rb^{4s-k-a+4b-2}}\int_{|\xi-\frac{-1-\sqrt{1+8\xi_{1}}}{2}| \leq 1}\frac{d\xi}{|\xi-\frac{-1-\sqrt{1+8\xi_{1}}}{2}|^{4b-1}} \\ & \lesssim \sup_{\xi_{1}}\frac{1}{\lb \xi_{1} \rb^{4s-k-a+4b-2}}.
\end{align*}

\noindent \textit{Case 5: $\xi_{1} \geq 1$, $|\xi-\frac{-1+\sqrt{1+8\xi_{1}}}{2}| \leq 1$ and $\frac{\sqrt{1+8\xi_{1}}-1}{4} \leq |\xi| \leq 1+\sqrt{1+8\xi_{1}}$.}

Similar to Case 4.

\noindent \textit{Case 6: $\xi_{1} \geq 1$, $\left |\xi-\frac{-1\pm \sqrt{1+8\xi_{1}}}{2}\right| \geq 1$ and $\frac{\sqrt{1+8\xi_{1}}-1}{4} \leq |\xi| \leq 1+\sqrt{1+8\xi_{1}}$.}

Note that in this case we have $|\xi| \approx |\xi_{1}|^{\frac{1}{2}}$. 

First we consider the subcase $\left |\xi-\frac{-1+\sqrt{1+8\xi_{1}}}{2} \right| \ll |\xi|$. In this case, we have
\begin{multline*}
     \lb \xi \rb^{4b-1}=\lb \xi \rb^{8b-3-} \lb \xi \rb^{2-4b+} \gtrsim \lb \xi \rb^{8b-3-}\lb \xi-\frac{-1+\sqrt{1+8\xi_{1}}}{2} \rb^{2-4b+} \\
     \gtrsim \lb \xi_{1} \rb^{\frac{1}{2}(8b-3-)}\lb \xi-\frac{-1+\sqrt{1+8\xi_{1}}}{2} \rb^{2-4b+}.
\end{multline*}
Using this,
\begin{align*}
    ( \ref{kdv_reduced} )
        & \lesssim \sup_{\xi_{1}}\frac{1}{\lb \xi_{1} \rb^{4s-k-a-1}}\int\frac{d\xi}{\lb \xi \rb^{4b-1}\lb \xi-\frac{-1+ \sqrt{1+8\xi_{1}}}{2}\rb^{4b-1}\lb \xi-\frac{-1- \sqrt{1+8\xi_{1}}}{2}\rb^{4b-1}} \\
        & \lesssim \sup_{\xi_{1}}\frac{1}{\lb \xi_{1} \rb^{4s-k-a+4b-\frac{5}{2}-}}\int\frac{d\xi}{\lb \xi-\frac{-1+ \sqrt{1+8\xi_{1}}}{2}\rb^{1+}\lb \xi-\frac{-1- \sqrt{1+8\xi_{1}}}{2}\rb^{4b-1}} \\
        & \lesssim \sup_{\xi_{1}}\frac{1}{\lb \xi_{1} \rb^{4s-k-a+4b-\frac{5}{2}-}\lb \sqrt{1+8\xi_{1}} \rb^{4b-1}} \\
        & \lesssim \sup_{\xi_{1}}\frac{1}{\lb \xi_{1} \rb^{4s-k-a+6b-3-}}.
\end{align*}

Next consider the subcase $\left |\xi-\frac{-1+\sqrt{1+8\xi_{1}}}{2} \right| \gtrsim |\xi|$. Noting that
\begin{align*}
    \lb \xi \rb^{2(4b-1)}=\lb \xi \rb^{2(4b-1)-1-}\lb \xi \rb^{1+} \gtrsim \lb \xi_{1} \rb^{(4b-1)-\frac{1}{2}-}\lb \xi \rb^{1+},
\end{align*}
we have
\begin{align*}
    ( \ref{kdv_reduced} )
        & \lesssim \sup_{\xi_{1}}\frac{1}{\lb \xi_{1} \rb^{4s-k-a-1}}\int\frac{d\xi}{\lb \xi \rb^{4b-1}\lb \xi-\frac{-1+ \sqrt{1+8\xi_{1}}}{2}\rb^{4b-1}\lb \xi-\frac{-1- \sqrt{1+8\xi_{1}}}{2}\rb^{4b-1}} \\
        & \lesssim \sup_{\xi_{1}}\frac{1}{\lb \xi_{1} \rb^{4s-k-a-1}}\int\frac{d\xi}{\lb \xi\rb^{2(4b-1)}\lb \xi-\frac{-1- \sqrt{1+8\xi_{1}}}{2}\rb^{4b-1}} \\
        & \lesssim \sup_{\xi_{1}}\frac{1}{\lb \xi_{1} \rb^{4s-k-a-4b-\frac{5}{2}-}}\int\frac{d\xi}{\lb \xi\rb^{1+}\lb \xi-\frac{-1- \sqrt{1+8\xi_{1}}}{2}\rb^{4b-1}}\\
        & \lesssim \sup_{\xi_{1}}\frac{1}{\lb \xi_{1} \rb^{4s-k-a-4b-\frac{5}{2}-}\lb \frac{-1- \sqrt{1+8\xi_{1}}}{2}\rb^{4b-1}} \\
        & \lesssim \sup_{\xi_{1}}\frac{1}{\lb \xi_{1} \rb^{4s-k-a-6b-3-}}.
\end{align*}

\subsection{Proof of Proposition \ref{estimate:uv correction}} \label{proof: uv correction}

It is enough to show that
\begin{align}
\sup_{\tau} \iiint \frac{{\chi_{R}\lb \tau + \xi^{2} \rb}^{s+a-\frac{3}{2}}d\xi_{1} d\tau_{1} d\xi}{\lb \xi - \xi_{1} \rb^{2s}\lb \xi_{1} \rb^{2k}\lb (\tau-\tau_{1})+(\xi - \xi_{1})^{2} \rb^{2b}\lb \tau_{1} - \xi_{1}^{3} \rb^{2b}} < \infty.\label{nlscorrection_eq1} 
\end{align}
By Lemma \ref{CalcLem}, we have 
\begin{align}
(\ref{nlscorrection_eq1}) \lesssim \sup_{\tau} \iint \frac{\chi_{R}{\lb \tau + \xi^{2} \rb}^{s+a-\frac{3}{2}}d\xi_{1} d\xi}{\lb \xi - \xi_{1} \rb^{s}\lb \xi_{1} \rb^{2k}\lb \tau-\xi_{1}^{3}-(\xi - \xi_{1})^{2} \rb^{4b-1}}.\label{nlscorrection_eq2} 
\end{align}
Since $\chi_{R}{\lb \tau + \xi^{2} \rb}^{s+a-\frac{3}{2}} \approx \chi_{R}{\lb \tau \rb}^{s+a-\frac{3}{2}}$, 
\begin{align*}
        (\ref{nlscorrection_eq2}) 
        &\lesssim \sup_{\tau} \int \frac{\lb \tau \rb^{s+a-\frac{3}{2}}d\xi_{1}}{\lb \xi_{1} \rb^{2k}}\int \frac{d\xi}{\lb \xi - \xi_{1} \rb^{2s}\lb \tau-\xi_{1}^{3}-(\xi - \xi_{1})^{2} \rb^{4b-1}} \\
        &\lesssim 
            \sup_{\tau} \int \frac{\lb \tau \rb^{s+a-\frac{3}{2}}d\xi_{1}}{\lb \xi_{1} \rb^{2k}} \\
            & \qquad \times \left(\int_{|\rho| \leq 1}\frac{d\rho}{\lb \rho \rb^{s}|\rho|^{\frac{1}{2}}\lb \tau-\xi_{1}^{3}-\rho \rb^{4b-1}}+\int_{|\rho|>1}\frac{d\rho}{\lb \rho \rb^{s}|\rho|^{\frac{1}{2}}\lb \tau-\xi_{1}^{3}-\rho \rb^{4b-1}}\right)  \\
        &\lesssim \sup_{\tau} \int \frac{\lb \tau \rb^{s+a-\frac{3}{2}}d\xi_{1}}{\lb \xi_{1} \rb^{2k}} \left(\int_{|\rho| \leq 1}\frac{d\rho}{|\rho|^{\frac{1}{2}}\lb \tau-\xi_{1}^{3} \rb^{4b-1}}+\int\frac{d\rho}{\lb \rho \rb^{s+\frac{1}{2}}\lb \tau-\xi_{1}^{3}-\rho \rb^{4b-1}}\right) \\
        &\lesssim \sup_{\tau} \int \frac{\lb \tau \rb^{s+a-\frac{3}{2}}d\xi_{1}}{\lb \xi_{1} \rb^{2k}} \left(\frac{1}{\lb \tau-\xi_{1}^{3} \rb^{4b-1}}+\frac{1}{\lb \tau-\xi_{1}^{3} \rb^{4b-1}}\right) \\
        &\lesssim \sup_{\tau} \lb \tau \rb^{s+a-\frac{3}{2}} \left(\int_{|\eta|<1} \frac{d\eta}{\lb \eta \rb^{\frac{2}{3}k}|\eta|^{\frac{2}{3}}\lb \tau-\eta \rb^{4b-1}}+\int_{|\eta|>1} \frac{d\eta}{\lb \eta \rb^{\frac{2}{3}k}|\eta|^{\frac{2}{3}}\lb \tau-\eta \rb^{4b-1}}\right) \\
        &\lesssim \sup_{\tau} \lb \tau \rb^{s+a-\frac{3}{2}} \left(\int_{|\eta|<1} \frac{d\eta}{|\eta|^{\frac{2}{3}}\lb \tau \rb^{4b-1}}+\int_{-\infty}^{\infty} \frac{d\eta}{\lb \eta \rb^{\frac{2}{3}k+\frac{2}{3}}\lb \tau-\eta \rb^{4b-1}}\right) \\
        &\lesssim \sup_{\tau} \lb \tau \rb^{s+a-\frac{3}{2}} \left(\frac{1}{\lb \tau \rb^{4b-1}}+\frac{1}{\lb \tau \rb^{\min(\frac{2}{3}k+\frac{2}{3},4b-1)}}\phi_{\max(\frac{2}{3}k+\frac{2}{3},4b-1)}(\tau)\right),
\end{align*}
where $\phi$ is as in Lemma \ref{CalcLem}. The last term is finite as long as $\frac{1}{2}<s<\min(\frac{2}{3}k+\frac{13}{6},\frac{5}{2})$ and $0\leq a <\min(\frac{2}{3}k-s+\frac{13}{6},\frac{5}{2}-s)$.

\subsection{Proof of Proposition \ref{estimate:vv_x correction}} \label{proof: vv_x correction}

We assume that $\frac32 <k< \frac72$ and $0 \leq a < \frac72 -k$. The case $\frac12 < k< 2$ is demonstrated in \cite{CT}.

By invoking Cauchy-Schwarz and Young's inequalities, we see that it suffices to show that
\begin{align*}
\sup_{\tau} \iiint \frac{\chi_R(\xi,\tau)\lb \tau - \xi^3 \rb^{\frac{2}{3}(k+a-2)}|\xi|^{2}d \xi_1 d \xi d \tau_1 }{\lb\xi_1 \rb^{2k} \lb \xi-\xi_1 \rb^{2k}\lb \tau_1 - \xi_1^3 \rb ^{2b}\lb \tau - \tau_1 - (\xi-\xi_1)^3 \rb^{2b}} < \infty.
\end{align*}
Using $\chi_R(\xi,\tau)\lb \tau - \xi^3 \rb^{\frac{2}{3}(k+a-2)} \approx \chi_R(\xi,\tau)\lb \tau \rb^{\frac{2}{3}(k+a-2)}$ and then applying Lemma \ref{CalcLem}, we can bound the right-hand side by
\begin{align} \label{eq: kdv correction reduced}
\sup_{\tau} \iint \frac{\lb \tau \rb^{\frac{2}{3}(k+a-2)}|\xi|^{2}d \xi_1 d \xi }{\lb\xi_1 \rb^{2k} \lb \xi-\xi_1 \rb^{2k}\lb \tau_1 - \xi_1^3-(\xi-\xi_1)^{3}\rb ^{4b-1}}.
\end{align}
Let $\rho=(\xi-\xi_1)^3$. Since $|\xi|^2 \lesssim |\rho|^{\frac13}+|\xi_1|^2$, we have
\begin{align*}
        (\ref{eq: kdv correction reduced}) 
        & \lesssim \sup_{\tau}\int\frac{\lb \tau \rb^{\frac{2}{3}(k+a-2)}}{\lb\xi_1 \rb^{2k}}\int\frac{|\rho|^{\frac13}+|\xi_1|^2}{\lb \rho \rb^{\frac{3k}{2}}|\rho|^{\frac23}\lb \tau_1 - \xi_1^3-\rho\rb ^{4b-1}} \\
        & \lesssim \sup_{\tau} \int \frac{\lb \tau \rb^{\frac{2}{3}(k+a-2)}d \xi_1}{\lb\xi_1 \rb^{2k-2}\lb \tau -\xi_1^3 \rb^{4b-1}} \\
        & \lesssim \sup_{\tau}\int \frac{\lb \tau \rb^{\frac{2}{3}(k+a-2)}d \eta}{\lb \eta \rb^{\frac{2k-2}{3}}|\eta|^{\frac23}\lb \tau -\eta \rb^{4b-1}} \\
        & = \sup_{\tau}\left ( \int_{|\eta| \leq 1} \frac{\lb \tau \rb^{\frac{2}{3}(k+a-2)}d \eta}{\lb \eta \rb^{\frac{2k-2}{3}}|\eta|^{\frac23}\lb \tau -\eta \rb^{4b-1}}+\int_{|\eta| > 1} \frac{\lb \tau \rb^{\frac{2}{3}(k+a-2)}d \eta}{\lb \eta \rb^{\frac{2k-2}{3}}|\eta|^{\frac23}\lb \tau -\eta \rb^{4b-1}}\right ) \\
        & \lesssim \sup_{\tau}\lb \tau \rb^{\frac{2}{3}(k+a-2)}\left ( \frac{1}{\lb \tau \rb^{4b-1}}+\int\frac{d\eta}{\lb \eta \rb^{\frac{2k}{3}}\lb \tau - \eta \rb^{4b-1}}\right ) \\
        & \lesssim \sup_{\tau}\lb \tau \rb^{\frac{2}{3}(k+a-2)-4b+1} < \infty.
\end{align*}

\subsection{Proof of Proposition \ref{estimate:|u|^2_x correction}} \label{proof: |u|^2_x correction}

By duality, it suffices to bound 
\begin{align} \label{eq: |u|^2_x correction dual}
    \iint_{\begin{subarray}{l}\\ \xi=\xi_{1}+\xi_{2}\\ \tau=\tau_{1}+\tau_{2} \end{subarray}} \chi_{R}\lb \tau -\xi^{3} \rb^{\frac{k+a-2}{3}}|\xi| |h(\tau)| \frac{|f(\xi_{1},\tau_{1})|}{\lb \xi_{1}\rb^{s}\lb \tau_{1}+\xi_{1}^{2} \rb^{b}}\frac{|g(\xi_{2},\tau_{2})|}{\lb \xi_{2}\rb^{s}\lb \tau_{2}-\xi_{2}^{2} \rb^{b}}
\end{align}
by $\Vert f \Vert_{L^{2}_{\xi,\tau}}\Vert g \Vert_{L^{2}_{\xi,\tau}}\Vert h \Vert_{L^{2}_{\tau}}$.

\noindent\textit{Case 1: $s \geq 0$ and $k<2$.}

Since $\lb \xi \rb^s \lesssim \lb \xi_{1} \rb^s \lb \xi_{2} \rb^s $, we can bound (\ref{eq: |u|^2_x correction dual}) by
\begin{align*}
        & \iint_{\begin{subarray}{l}\\ \xi=\xi_{1}+\xi_{2}\\ \tau=\tau_{1}+\tau_{2} \end{subarray}} \chi_{R}\lb \tau -\xi^{3} \rb^{\frac{k+a-2}{3}}|\xi|\lb \xi \rb^{-s}|h(\tau)| \frac{\lb \xi_1 \rb^{s}|f(\xi_{1},\tau_{1})|}{\lb \tau_{1}+\xi_{1}^{2} \rb^{b}}\frac{|g(\xi_{2},\tau_{2})|}{\lb \xi_{2}\rb^{s}\lb \tau_{2}-\xi_{2}^{2} \rb^{b}} \\
        &= \iint_{\begin{subarray}{l}\\ \xi=\xi_{1}+\xi_{2}\\ \tau=\tau_{1}+\tau_{2} \end{subarray}} \lb \xi \rb^{\frac12 -}\left (\chi_{R}\lb \tau -\xi^{3} \rb^{\frac{k+a-2}{3}}|\xi|\lb \xi \rb^{-s-\frac12+}|h(\tau)| \right )\frac{|f(\xi_{1},\tau_{1})|}{\lb \tau_{1}+\xi_{1}^{2} \rb^{b}}\frac{|g(\xi_{2},\tau_{2})|}{\lb \tau_{2}-\xi_{2}^{2} \rb^{b}} \\
        & \lesssim \| \chi_{R}\lb \tau -\xi^{3} \rb^{\frac{k+a-2}{3}}|\xi|\lb \xi \rb^{-s-\frac12+}|h(\tau)|\|_{L^{2}_{\xi,\tau}}\Vert g \Vert_{L^{2}_{\xi,\tau}}\Vert h \Vert_{L^{2}_{\tau}} \\
        & \lesssim \sup_{\tau}\left( \lb \tau \rb^{\frac{k+a-2}{3}} \Vert \chi_{R}(\xi,\tau)|\xi|\lb \xi \rb^{-s-\frac12+} \Vert_{L^{2}_{\xi}}\right)\Vert f \Vert_{L^{2}_{\xi,\tau}}\Vert g \Vert_{L^{2}_{\xi,\tau}}\Vert h \Vert_{L^{2}_{\tau}}.
\end{align*}
Since
\begin{align*}
    \sup_{\tau}\left( \lb \tau \rb^{\frac{k+a-2}{3}} \Vert \chi_{R}(\xi,\tau)|\xi|\lb \xi \rb^{-s-\frac12+} \Vert_{L^{2}_{\xi}}\right) \lesssim \sup_{\tau} \lb \tau \rb^{\frac{k+a-2}{3}} \lb \tau \rb^{\frac{1}{3}\max (1-s,0)+},
\end{align*}
the desired bound follows provided $k+a<2$ and $k+a<s+1$.

\noindent\textit{Case 2: $\frac12 < s <\frac32$.}

It suffices to show that the quantity
\begin{align}
\sup_{\tau} \iiint \frac{{\chi_{R}\lb \tau + \xi^{2} \rb}^{\frac{2}{3}(k+a-2)}|\xi|^2d\xi_{1} d\tau_{1} d\xi}{\lb \xi - \xi_{1} \rb^{2s}\lb \xi_{1} \rb^{2s}\lb (\tau-\tau_{1})+(\xi - \xi_{1})^{2} \rb^{2b}\lb \tau_{1} - \xi_{1}^{3} \rb^{2b}} \label{kdvcorrection_eq1} 
\end{align}
is finite.

Using $\chi_{R}{\lb \tau - \xi^{3} \rb}^{\frac{2}{3}(k+a-2)} \approx \chi_{R}{\lb \tau \rb}^{\frac{2}{3}(k+a-2)}$ and then applying Lemma \ref{CalcLem}, we have 
\begin{align}
(\ref{kdvcorrection_eq1}) \lesssim \sup_{\tau} \iint \frac{{\lb \tau \rb}^{\frac{2}{3}(k+a-2)} |\xi|^2 d\xi_{1} d\xi}{\lb \xi - \xi_{1} \rb^{s}\lb \xi_{1} \rb^{2k}\lb \tau-\xi_{1}^{3}-(\xi - \xi_{1})^{2} \rb^{4b-1}}.\label{kdvcorrection_eq2} 
\end{align}
Let $\rho=(\xi-\xi_1)^2$. Since $|\xi|^2 \lesssim |\rho|+|\xi_1|^2$, 
\begin{align*}  
        (\ref{kdvcorrection_eq2}) 
        &\lesssim \sup_{\tau} \int \frac{\lb \tau \rb^{\frac{2}{3}(k+a-2)} d\xi_1}{\lb \xi_{1} \rb^{2s}}\int \frac{|\rho|+|\xi_1|^2 d\rho}{\lb \rho \rb^{s}|\rho|^{\frac{1}{2}}\lb \tau-\xi_{1}^{2}-\rho \rb^{4b-1}} \\
        &\lesssim 
            \sup_{\tau} \int \frac{\lb \tau \rb^{\frac{2}{3}(k+a-2)}d\xi_{1}}{\lb \xi_{1} \rb^{2s}} \Bigg(\int\frac{d\rho}{\lb \rho \rb^{s-\frac12}|\rho|^{\frac{1}{2}}\lb \tau-\xi_{1}^{2}-\rho \rb^{4b-1}} \\ 
            & \qquad +\int_{|\rho|\leq 1}\frac{d\rho}{\lb \rho \rb^{s}|\rho|^{\frac{1}{2}}\lb \tau-\xi_{1}^{2}-\rho \rb^{4b-1}}+\int_{|\rho|>1}\frac{d\rho}{\lb \rho \rb^{s+\frac12}\lb \tau-\xi_{1}^{2}-\rho \rb^{4b-1}}\Bigg) \\
        &\lesssim \sup_{\tau} \int \frac{\lb \tau \rb^{\frac{2}{3}(k+a-2)}d\xi_{1}}{\lb \xi_{1} \rb^{2s}} \left(\frac{1}{\lb \tau-\xi_1^2 \rb^{s+4b-\frac52}}+\frac{\lb \xi_1 \rb^2}{\lb \tau-\xi_{1}^{2}\rb^{4b-1}}\right) \\
        &\lesssim \sup_{\tau} \lb \tau \rb^{\frac{2}{3}(k+a-2)} \left ( \int \frac{d\eta}{\lb \eta \rb^{s}|\eta|^{\frac12}\lb \tau-\eta \rb^{s+4b-\frac52}}+\int \frac{\lb \eta \rb d\eta}{\lb \eta \rb^{s}|\eta|^{\frac12}\lb \tau-\eta \rb^{4b-1}} \right ) \\
        &\lesssim \sup_{\tau} \lb \tau \rb^{\frac{2}{3}(k+a-2)} 
            \Bigg ( \int_{|\eta| \leq 1} \frac{d\eta}{\lb \eta \rb^{s}|\eta|^{\frac12}\lb \tau-\eta \rb^{s+4b-\frac52}}+\int \frac{d\eta}{\lb \eta \rb^{s+\frac12}\lb \tau-\eta \rb^{s+4b-\frac52}} \\
            & \qquad +\int_{|\eta| \leq 1} \frac{d\eta}{\lb \eta \rb^{s-1}|\eta|^{\frac12}\lb \tau-\eta \rb^{4b-1}}+\int \frac{d\eta}{\lb \eta \rb^{s-\frac12}\lb \tau-\eta \rb^{4b-1}} \Bigg )\\
        &\lesssim \sup_{\tau} \lb \tau \rb^{\frac{2}{3}(k+a-2)} 
            \Bigg (\frac{1}{\lb \tau \rb^{s+4b-\frac52}} 
            +\frac{1}{\lb \tau \rb^{4b-1}}  \Bigg ) \\
        &\lesssim \sup_{\tau} \lb \tau \rb^{\frac{2}{3}(k+a-2)-s-4b+\frac52}. 
\end{align*}
The last term is finite as long as $k+a<\frac{3}{2}s+\frac54$. 

\noindent \textit{Case 3: $\frac32 \leq s <\frac52$.}

We can proceed as in Case 2. In this case, (\ref{kdvcorrection_eq2}) is finite provided $k+a<\frac72$.


\section{Appendix}

\begin{lem} {\cite[Lemma A.2]{ET}} \label{CalcLem} 
If $\beta \geq \gamma \geq 0$ and $\beta+\gamma>1$, then
\begin{align*}
  \int\frac{dx}{\lb x-a_1 \rb^{\beta}\lb x-a_2 \rb^{\gamma}} \lesssim \lb a_1-a_2 \rb^{-\gamma}\phi_{\beta}(a_1-a_2),
\end{align*}
where
\begin{align*}
\phi_{\beta}(a)=
  \begin{cases}
  1 & \beta>1 \\
  \log (1+\lb a \rb) & \beta =1 \\
  \lb a \rb^{1-\beta} & \beta <1.
  \end{cases}
\end{align*}
\end{lem}

\end{document}